\documentclass[12pt,reqno, a4paper]{amsart}
\usepackage{hyperref}

\usepackage{amsfonts, amssymb,amsmath,amsthm,enumerate, mathtools}

\usepackage[margin=2.5cm]{geometry}

\usepackage{bbm}

\newcommand{\beq}{\begin{equation}}
\newcommand{\eeq}{\end{equation}}

\newcommand{\myboldfont}{\mathbb}

\newcommand{\eps}{\varepsilon}
\renewcommand{\epsilon}{\varepsilon}
\renewcommand{\phi}{\varphi} 
\renewcommand{\kappa}{\varkappa}
\renewcommand{\ge}{\geqslant}
\renewcommand{\le}{\leqslant}
\renewcommand{\leq}{\leqslant}
\renewcommand{\geq}{\geqslant}
\renewcommand{\emptyset}{\varnothing}

\newcommand{\abcd}{\ensuremath{\begin{pmatrix}a & b\\c& d\end{pmatrix}}}
\newcommand{\smallabcd}{\ensuremath{\left(\begin{smallmatrix}a & b\\c& d\end{smallmatrix}\right)}}

\makeatletter   
\let\th\@undefined               
\makeatother

\DeclareMathOperator{\th}{th}

\DeclareMathOperator{\ctg}{ctg}
\DeclareMathOperator{\leb}{Leb}

\DeclareMathOperator{\supp}{supp}

\newcommand{\sltr}{\ensuremath{\mathrm{SL}(2,\myboldfont R)}}
\newcommand{\sltz}{\ensuremath{\mathrm{SL}(2,\myboldfont Z)}}
\newcommand{\SL}{\ensuremath{\mathrm{SL}}}

\newcommand{\Sl}{\ensuremath{\mathfrak{sl}}}

\newcommand{\quot}{\ensuremath{\backslash}}

\newcommand{\ASL}{\ensuremath{\mathrm{ASL}}}

\newcommand{\R}{\ensuremath{\myboldfont R}}

\newcommand{\Z}{\ensuremath{\myboldfont Z}}

\newcommand{\T}{\ensuremath{\myboldfont T}}

\newcommand{\N}{\ensuremath{\myboldfont N}}

\renewcommand{\d}{\ensuremath{\partial}}

\renewcommand{\myboldfont}{\mathbb}

\makeatletter
\def\imod#1{\allowbreak\mkern5mu{\operator@font mod}\,\,#1}

\makeatother

\newcommand{\sumstar}{\sideset{}{^\star} \sum}
\newcommand{\beqq}{\begin{equation*}}
\newcommand{\eeqq}{\end{equation*}}
\newcommand{\bal}{\begin{align}}
\newcommand{\eali}{\end{align}}
\newcommand{\ball}{\begin{align*}}
\newcommand{\ealii}{\end{align*}}

\newcommand{\FF}{\mathbb{F}}
\newcommand{\NN}{\mathbb{N}}
\newcommand{\ZZ}{\mathbb{Z}}

\renewcommand{\leq}{\leqslant}
\renewcommand{\le}{\leqslant}
\renewcommand{\geq}{\geqslant}
\renewcommand{\ge}{\geqslant}
\DeclareMathOperator{\ord}{ord} 
\DeclareMathOperator{\Sing}{Sing}
\newcommand{\1}{\mathbf{1}}

\usepackage{bm}
\usepackage{enumerate}

\newtheorem{theorem}{Theorem}[section]
\newtheorem{lemma}[theorem]{Lemma}
\newtheorem{prop}[theorem]{Proposition}
\newtheorem{cor}[theorem]{Corollary}

\theoremstyle{definition}
\newtheorem*{ack}{Acknowledgements}
\newtheorem*{notation}{Notation}

\numberwithin{equation}{section}

\newcommand{\ind}[1]{\ensuremath{{\mathbbm{1}}{\big(#1\big)}}}

\begin{document}
\title[Effective Ratner theorem for $\SL(2,\R)\ltimes \R^2$]{Effective Ratner theorem for $\SL(2,\R)\ltimes \R^2$\\ and gaps in $\sqrt n$ modulo $1$}

\author{Tim Browning}
\address{School of Mathematics\\
University of Bristol\\ Bristol\\ BS8 1TW\\ United Kingdom}
\email{t.d.browning@bristol.ac.uk}
\author{Ilya Vinogradov}
\address{Department of Mathematics\\
Princeton University \\ Princeton, NJ\\ 08544 \\ United States}
\email{ivinogra@math.princeton.edu}

\begin{abstract}
Let 
$G=\SL(2,\R)\ltimes \R^2$ and $\Gamma=\SL(2,\Z)\ltimes \Z^2$.
Building on recent work of Str\"ombergsson we prove a rate of equidistribution for the orbits of a 
certain 
$1$-dimensional unipotent flow
of $\Gamma\quot G$, which projects 
to a closed horocycle in the unit tangent bundle to the modular surface.  
We use this to answer a question of 
Elkies and McMullen by making effective the 
convergence of the gap distribution of $\sqrt n \imod 1$.
\end{abstract}

\subjclass[2010]{37A25 (11L07, 11J71, 37A17)}

\maketitle

\thispagestyle{empty}

\section{Introduction}

Results  of Ratner on measure rigidity and equidistribution of orbits \cite{ratner_raghunathans_1991, ratner_raghunathans_1991_1} play a fundamental role in the study of unipotent flows on homogeneous spaces. They have many applications beyond the world of dynamics, ranging from problems in number theory to mathematical physics.
This paper is concerned with the problem of obtaining {\em effective} versions of results that build on Ratner's theorem and is inspired by recent work of Str\"ombergsson~\cite{strombergsson_effective_2013}.

Let 
$G=\ASL(2,\R)=\SL(2,\R)\ltimes\R^2$ be the group of affine linear transformations of $\R^2$. 
We define the product on $G$ by 
\[(M,\bm x)(M',\bm x')=(MM',\bm xM'+\bm x'),\]
and the right action is given by $\bm x(M,\bm x')=\bm x M+ \bm x'.$ We always think of $\bm x\in\R^2$ as a row vector.
Put  $\Gamma=\ASL(2,\Z)=\SL(2,\Z)\ltimes\Z^2$ and let 
$X=\Gamma \quot G$ be the associated homogeneous space. 
The group $G$ is unimodular and so the Haar measure $\mu$ on $G$ projects to a right-invariant measure on $X$. 
The space $X$ is non-compact,  but it has finite volume with respect to the projection of $\mu.$ 
Following the  usual abuse of notation, we denote the projected measure by $\mu$ and normalize it so that $\mu(X)=1$.

Let 
$$
a(y)=\begin{pmatrix}\sqrt y & 0\\ 0&1/\sqrt y\end{pmatrix},
$$
and write 
 $A^+ = \{a(y)\colon y>0\}$.
In what follows we will use the embedding 
 $\sltr \hookrightarrow G$, given  by  $M\mapsto (M,\bm 0)$, which thereby  allows us to think of \sltr\ as a subgroup of $G$.
Str\"ombergsson \cite{strombergsson_effective_2013} works with the unipotent flow on $X$ generated by  right multiplication by the subgroup 
\[
U_0=\left \{\left(\begin{pmatrix}1&x\\0 &1\end{pmatrix},\left(0,0\right)\right)\colon x\in\R\right \}.
\] 
He considers orbits of a point $(\mathrm{Id}_{2},(\xi_1,\xi_2))$ subject to a certain Diophantine condition. 
In \cite[Thm.~1.2]{strombergsson_effective_2013}, effective rates of convergence are obtained for the equidistribution of such orbits under the flow $a(y)$ as $y\to 0$. The goal of the present paper is to 
extend the methods of 
 Str\"ombergsson to handle the  orbit 
 generated by  right multiplication by the subgroup 
$U=\{u(x): x\in \R\}$, where
\[u(x)=\left(\begin{pmatrix}1&x\\0&1\end{pmatrix},\left(\frac x2,\frac{x^2}4\right)\right).\]
As noted by 
Str\"ombergsson \cite[\S 1.3]{strombergsson_effective_2013}, {\em any} Ad-unipotent $1$-parameter subgroup of $G$ with non-trivial image in $\SL(2,\R)$ is conjugate to either  $U_0$ or $U$.

With this notation we will  study the rate of equidistribution of the closed orbit $\Gamma\quot \Gamma U$ under the action of $a(y)$, as $y\to 0$. Geometrically this orbit is a lift of a closed horocycle in $\sltz\quot \sltr$ to $\Gamma\quot G$, and the behaviour of horocycles under the flow $A^+$ on $\sltz\quot \sltr$ is very well understood.
The main obstruction to treating the problem of horocycle lifts with the usual techniques of ergodic theory (such as thickening) is the fact that $U$ is \emph{not} the entire unstable manifold of the flow $a(y)$, but only a  codimension 1 submanifold. Elkies and McMullen \cite{ElkiesMcM04} used Ratner's measure classification theorem \cite{ratner_raghunathans_1991} to prove that the horocycle lifts  equidistribute, but their method is ineffective. 
In \cite[\S 3.6]{ElkiesMcM04} they ask whether explicit error estimates can be obtained. The following result answers this   affirmatively.

\begin{theorem}\label{th:main}
There exists $C>0$ such that for every $f\in C_{\mathrm{b}}^8(X)$ and $y>0$ we have 
$$
\left|\frac12 \int_{-1}^1f(u(x)a(y))\, dx-\int_X  f\, d\mu\right| < C\|f\|_{C_{\mathrm{b}}^8} y^{\frac14}\log^2 (2+y^{-1}).
$$
\end{theorem}
Here $C_{\mathrm{b}}^k(X)$ denotes the space of $k$ times continuously differentiable functions on $X$ whose left-invariant derivatives up to order $k$ are bounded (see equation \eqref{eq:norm} for the exact definition of the norm).

\medskip

Our next result shows that we  can replace $dx$ by a sufficiently smooth absolutely continuous measure. Let $\rho\colon \R\to \R_{\ge 0}$ be a compactly supported function that has $1+\eps$ derivatives in $L^1$. For simplicity we follow \cite{strombergsson_effective_2013} and interpolate between the Sobolev norms $\|\rho\|_{W^{1,1}}$ and $\|\rho\|_{W^{2,1}}$, which give the $L^1$ norms of first and second derivatives, respectively. This interpolation allows us to treat the case of  piecewise constant functions with an $\eps$-loss in the rate.

\begin{theorem}
\label{th:main_second}
Let  $\eta\in (0,1)$.
There exists $K>1$  and  $C(\eta)>0$ such that for every $f\in C_{\mathrm{b}}^8(X)$ and $y>0$ we have
$$
\left| \int_\R 
\hspace{-0.2cm}
f(u(x)a(y))\rho(x) dx-\int_X  
\hspace{-0.1cm}
f\, d\mu\int_{\R}
\hspace{-0.1cm}
\rho(x)\, dx \right| < C(\eta)\|\rho\|_{W^{1,1}}^{1-\eta} \|\rho\|_{W^{2,1}}^\eta \|f\|_{C_{\mathrm{b}}^8} y^{\frac14} \log^{K-1}(2+y^{-1}).
$$
\end{theorem}

The constant $K$ in this result is absolute and does not depend on $\eta$.
The proof of
Theorems 
\ref{th:main}
 and  \ref{th:main_second} 
 builds on the proof of \cite[Thm.~1.2]{strombergsson_effective_2013}. It relies on Fourier analysis and estimates for complete exponential sums which are essentially due to Weil. Let us remark that while we strive to obtain the best possible decay in $y$, we take little effort to optimize the norms of $f$ and $\rho$ that appear in the estimates. 
 The exponent $\frac14$ in the error term is optimal for our method, but we surmise it can be improved by exploiting additional cancellation in certain two dimensional exponential sums. 
 The natural upper limit is $\frac12$, which holds for horocycles on $\SL(2,\Z)\quot \SL(2,\R)$ due to work of Sarnak \cite{sarnak_asymptotic_1981}.

\medskip

We may apply Theorem \ref{th:main} to study gaps between the fractional parts of $\sqrt{n}$. 
Consider the sequence $\sqrt n\imod 1 \subset \R/\Z\cong  S^1$. It is easy to see from Weyl's criterion that this sequence is uniformly distributed on the circle. This means that 
for every interval $J\subset S^1$, we have 
\[
 \lim_{N\to\infty} \frac{\#\{\sqrt n\imod 1 \colon 1\le n \le N\}\cap J}{N} = |J|,
\]
where $|\cdot|$ denotes length. The statistic we focus on is the \emph{gap distribution}. For each $N\in \N$, we consider the set $\{\sqrt n\imod 1\}_{1\le n\le N}$ and we allow $0\in \R/\Z$ to be included for each perfect square. This set of $N$ points divides the circle into $N$ intervals (a few of which could be of zero length) which we refer to as \emph{gaps}. For $t\ge 0$, we define the \emph{gap distribution} $\lambda_N(t)$ to be the proportion of gaps whose length is less than $t/N$. This function satisfies $\lambda_N(0)=0$ and $\lambda_N(\infty)=1$, and it is left-continuous.

The behaviour of $\lambda_N(t)$, as $N\to\infty$, has been  analyzed by Elkies and McMullen \cite{ElkiesMcM04} and later also by Sinai \cite{Sinai13}. It is shown in \cite{ElkiesMcM04} that there exists a function $\lambda_\infty(t)$ such that $\lambda_N(t)\to\lambda_\infty(t)$ for each $t$. 
We have  
$$\lambda_\infty(t)=\int_0^t F(\xi)\,d\xi,
$$
where $F$ is given in 
\cite[Thm.~1.1]{ElkiesMcM04}. 
It is defined by analytic functions on three intervals, but it is not analytic at the endpoints joining these intervals.
Moreover, 
it is constant on the interval $[0,1/2].$ 

% The approach of Elkies and McMullen \cite{ElkiesMcM04} is to relate $\lambda_N(t)$ to a function on $X$, so that the problem of understanding $\lambda_N(t)$ is translated into studying 
% \[
% \frac12 \int_{-1}^1 f(u(x)a(1/N))\, dx,
% \]
% as $N\to \infty$, for a certain function $f$ that depends on $t$.  The error terms appearing in this step are worked out explicitly in \cite{ElkiesMcM04}. 
% In fact, $f$ is directly related to 
% \beq\label{eq:sigma_def}
% \sigma_N(t)=\int_0^t\xi d\lambda_N(\xi),
% \eeq
% which is the total length of gaps whose length is less than $ t/N$. 

The key input in \cite{ElkiesMcM04} comes from 
Ratner's theorem \cite{ratner_raghunathans_1991}, which is used in
\cite[Thm.~2.2]{ElkiesMcM04}  
to find the limiting distribution of $\lambda_N(t)$ and therefore cannot give a rate of convergence. 
Armed with Theorem \ref{th:main}, we will refine this approach to get the following result.

\begin{cor}\label{cor:sqrtn}%T
Let 
$ \lambda_N(t), \lambda_\infty(t)$ be as above. Then for every $\eps>0$, there exists $C_\eps>0$ such that
%T [who cares about \log ^{1/5} vs the weaker \log... also different function in front than the one we discussed] 
\[
\left| \lambda_N(t) - \lambda_\infty (t)\right|   < C_\eps  (1/t^2+t) N^{-\frac{1}{68} + \eps}
\]
for any $N\geq 2$ and $t>0$. 
\end{cor}

\medskip

The sequence $\sqrt n\bmod 1$ has also been studied from the perspective of  its \emph{pair correlation function}.
This is a useful  statistic for measuring randomness in sequences and, in this setting,  it has been shown to converge to that of a Poisson point process by El-Baz, Marklof, and the second author \cite{el-baz_two-point_2013}. In the light of Theorem \ref{th:main}, 
although we will not carry out the details here, by developing effective versions of the results in 
\cite{el-baz_two-point_2013} it would be possible to
conclude that the pair correlation function converges \emph{effectively}. 
By way of comparison, we remark that  
Str\"ombergsson \cite[\S 1.3]{strombergsson_effective_2013} indicates how one might make effective 
the convergence of the pair correlation function in the problem of {\em visible lattice points} (see \cite{EMV_directions_2013}).

\medskip

The plan of the paper is as follows. 
In Section \ref{sec:fourier}, we embark on the proof of Theorem~\ref{th:main} by developing  $f$ into a Fourier series in the torus coordinate. Section \ref{sec:exponential} is dedicated to estimating certain complete exponential sums that are required in Section \ref{sec:errorterms} to control the error terms.
Corollary \ref{cor:sqrtn} is proved in Section \ref{sec:sqrtn} and, finally,  the proof of Theorem \ref{th:main_second} is sketched in Section \ref{sec:generaldensity}.

\begin{notation}
Given functions $f,g:S\rightarrow \R$, with 
$g$ positive,  we will write $f\ll g$ if 
there exists a constant $c$ such that 
$|f(s)|\leq cg(s)$ for all $s\in S$.
\end{notation}

\begin{ack}
The research leading to these results has received funding from the European Research Council under the European Union's Seventh Framework Programme (FP/2007--2013) / ERC Grant Agreements  \texttt{291147} and  \texttt{306457}. The authors are grateful to Andreas Str\"ombergsson and Jens Marklof for helpful discussions and comments on an earlier draft. Special thanks are due to the  referee for numerous  useful remarks that have improved the paper considerably.  
\end{ack}

\section{Fourier Decomposition\label{sec:fourier}}

In this section we develop the tools necessary to prove Theorem \ref{th:main} and decompose $f$ into a Fourier series on the torus. We proceed exactly as in   \cite{strombergsson_effective_2013}. 
To begin with we note that 
\[f((1,\bm \xi)M)=f((1,\bm \xi+\bm n)M)\] 
for $\bm n \in \Z^2$. 
So for $M$ fixed, $f$ is a well defined function on $\R^2/\Z^2$ and we can expand it into a Fourier series as 
\beq \label{eq:fourier}f((1,\bm \xi)M)=\sum_{\bm m\in\Z^2}\hat f(M,\bm m) e(\bm m.\bm \xi),\eeq
where
$$ \hat f(M,\bm m)=\int_{\T^2}f((1,\bm \xi')M)e(-\bm m .\bm \xi')d\bm \xi'.
$$
Note that 
\beq\hat f(TM,\bm m)=\hat f(M,\bm m (T^{-1})^t),
\label{eq:transpose}\eeq
for $T\in\sltz.$ Set $\tilde f_n(M)=\hat f(M,(n,0))$.  
These functions of $M\in\sltr$ are left-invariant under the group $\left(\begin{smallmatrix}1&\Z\\0&1\end{smallmatrix}\right)$ by \eqref{eq:transpose}.

Now it follows from 
 \eqref{eq:transpose}  that
\begin{align*}
 \tilde f_n\left(\abcd M\right)=\hat f \left(\abcd M,(n,0)\right)
 &=\hat f \left(M,(n,0)\begin{pmatrix} d & -c\\ -b & a\end{pmatrix}\right)\\
 &=\hat f(M,(nd,-nc)).
\end{align*}
Therefore we can rewrite \eqref{eq:fourier} with $\bm \xi= (x/2,-x^2/4)$ as 
\beq f\left(\left(1,\left(x/2,-x^2/4\right)\right)M\right)=\tilde f_0(M)+\sum_{n\ge 1}\sum_{(c,d)=1} \tilde f_n\left(\begin{pmatrix}*&*\\c&d\end{pmatrix}M\right)e\left(n\left(\frac{dx}2 + \frac{cx^2}{4}\right)\right),\label{eq:afterfourier}\eeq
where $\left(\begin{smallmatrix}*&*\\c&d\end{smallmatrix}\right)=\smallabcd$ is any matrix in \sltz\ with $c$ and $d$ in the second row as specified. 

Integrating  \eqref{eq:afterfourier} over $x$, we obtain
$$\frac12\int_{-1}^1 f(u(x)a(y))dx=
M(y)+E(y), 
$$
where 
\beq\label{eq:main_term}
M(y)=
\frac12 \int_{-1}^1 \tilde f_0\begin{pmatrix}\sqrt y&x/\sqrt y\\0&1/\sqrt y\end{pmatrix} dx
\eeq
and 
\beq
E(y)=\sum_{\substack{n\ge1\\ (c,d) =1}}\frac12 \int_{-1}^1 e\left(n\left(\frac{dx}2+\frac{cx^2}{4}\right)\right)\tilde f_n\left(\begin{pmatrix}*&*\\c&d\end{pmatrix} \begin{pmatrix}\sqrt y&x/\sqrt y\\0&1/\sqrt y\end{pmatrix}\right)dx.\label{eq:errorterms}
\eeq
The main term in this expression is $M(y)$ and, as  is well-known 
(cf.~\cite{sarnak_asymptotic_1981, flaminio_invariant_2003, burger_horocycle_1990, strombergsson_effective_2013}), we have 
\[M(y)=\int_X f\, d\mu+O(\|f\|_{C_{\mathrm{b}}^4} y^{1/2-\eps}).\]
This statement is nothing more than effective equidistribution of horocycles under the geodesic flow on $\sltz\quot\sltr.$ We need not seek the best error term for this problem, since there will be larger contributions to the error term in Theorem \ref{th:main}.

It remains to estimate $E(y)$  as $y\to 0$, which we do in Section \ref{sec:errorterms}. 

\medskip

We end this section with a pair of technical results that will help us to estimate $E(y)$. First, however, we give a precise definition  of $\|\cdot\|_{C_{\mathrm{b}}^m}$ for functions on $G$ and hence also on $X$. Following \cite{strombergsson_effective_2013}, we let $\mathfrak g = \Sl(2,\R)\oplus  \R^2$ be the Lie algebra of $G$ and fix  
\begin{align}\label{eq:basis}
\begin{aligned}
X_1&=\left(\left(\begin{smallmatrix}0&1\\0&0\end{smallmatrix}\right),\bm 0\right),\quad 
X_2=\left(\left(\begin{smallmatrix}0&0\\1&0\end{smallmatrix}\right),\bm 0\right),\quad
X_3=\left(\left(\begin{smallmatrix}1&0\\0&-1\end{smallmatrix}\right),\bm 0\right),\\
X_4&=\left(\left(\begin{smallmatrix}0&0\\0&0\end{smallmatrix}\right),(1,0)\right),\quad
X_5=\left(\left(\begin{smallmatrix}0&0\\0&0\end{smallmatrix}\right),(0,1)\right)
\end{aligned}
\end{align}
to be a basis of $\mathfrak g$. Every element of the universal enveloping algebra $U(\mathfrak g)$ corresponds to a left-invariant differential operator  on functions on $X$.  We define
\beq\label{eq:norm}
\|f\|_{C_{\mathrm{b}}^m} = \sum_{\deg D\le m} \|Df\|_{L^\infty},
\eeq
where the sum runs over monomials in $X_1,\dots,X_5$ of degree at most $m$. 

The following result is
  \cite[Lemma~4.2]{strombergsson_effective_2013}.

\begin{lemma}\label{lem:andreas_fourier1}
Let $m\ge0$ and $n>0$ be integers. Then 
\[\tilde f_n\abcd \ll_m \frac{\|f\|_{C_{\mathrm{b}}^m}}{n^m(c^2+d^2)^{m/2}}, \quad \forall \abcd \in \SL(2,\R). \]
\end{lemma}

Passing to Iwasawa coordinates in \sltr, we  write 
\beq\label{eq:iwasawa}\tilde f_n(u,v,\theta)=\tilde f_n\left(
 \begin{pmatrix}
  1&u\\
  0&1
 \end{pmatrix}
\begin{pmatrix}
 \sqrt v&0\\
 0&1/\sqrt v
\end{pmatrix}
\begin{pmatrix}
 \cos\theta &-\sin\theta\\
 \sin\theta &\cos\theta
\end{pmatrix}\right).
\eeq
for $u\in \mathbb{R}, v>0$ and $\theta\in \mathbb{R}/2\pi\mathbb{R}$.
The following  is 
  \cite[Lemma~4.4]{strombergsson_effective_2013}.

\begin{lemma}\label{lem:andreas_fourier2}
Let $m, k_1, k_2, k_3\ge 0$ and $n>0$ be integers, and let $k=k_1+k_2+k_3$. Then
\[\d_u^{k_1}\d_v^{k_2} \d_\theta^{k_3} \tilde f_n(u,v,\theta)\ll_{m,k} \|f\|_{C_{\mathrm{b}}^{m+k}}n^{-m} v^{m/2-k_1-k_2}.\]
\end{lemma}

\section{Complete exponential sums\label{sec:exponential}}

In this section we make a detailed examination of the 
exponential sum
\begin{equation}\label{eq:Tq}
T_q(A,B)=
\sum_{\substack{n\imod{q} \\ (n,q)=1}} e_q\left(
An^2+B\bar n\right),
\end{equation}
for $A,B\in \ZZ$ and   $q\in \NN$.
Here $e_q(\cdot)=e(\frac{\cdot}{q})$ and $\bar n$ is the multiplicative inverse of $n$ modulo $q$.
Our main tool is  Weil's resolution of the Riemann hypothesis for function fields in one variable (see Bombieri \cite{bombieri}), together with some general results due to Cochrane and Zheng \cite{cochrane-zheng} about exponential sums  
involving rational functions and  higher  prime power moduli.  The procedure we follow is very general and could  easily be adapted to handle other exponential sums of similar type.

We begin by recording the easy multiplicativity property
\begin{equation}\label{eq:mult}
T_{q_1q_2}(A,B)=T_{q_1}(\bar q_2A, \bar q_2 B)
T_{q_2}(\bar q_1A, \bar q_1B),
\end{equation}
whenever  $q_1,q_2\in \NN$ are coprime and $\bar q_1,\bar q_2\in \ZZ$ satisfy $q_1\bar q_1+q_2\bar q_2=1.$
This renders it sufficient to study $T_{p^m}(A,B)$ for a prime power $p^m$.
We may write $T_{p^m}(A,B)$
in the form
\begin{equation}\label{eq:fall}
\sumstar_{\substack{n\imod{p^m}}} e_{p^m}\left(\frac{f_1(n)}{f_2(n)}
\right),
\end{equation}
where $f_1(x)=Ax^3+B$ and $f_2(x)=x$.
The symbol 
 $\sum^\star$ means that $n$ is only taken over values for which  $p\nmid f_2(n)$, 
 in which scenario  $f_1(n)/f_2(n)$ means 
$f_1(n)\overline{f_2(n)}$.
We proceed by establishing the following result, which deals with the odd prime powers.

\begin{lemma}\label{lem2}
Let $p>2$ and $m\in \NN$. Then we have 
$$
|T_{p^m}(A,B)|
\leq 
\begin{cases}
3 p^{m/2}(p^m,A,B)^{1/2}, &\mbox{if $p>3$,}\\
3^{1+3m/4}(3^m,A,B)^{1/4}, &\mbox{if $p=3$.}
\end{cases}
$$
\end{lemma}

\begin{proof}
When $m=1$ the sum in which we are interested is a classical exponential sum over a finite field and we may use the Weil bound, in the form developed by Bombieri \cite{bombieri} for rational functions.
This leads to the satisfactory estimate 
\begin{equation}\label{eq:weil}
|T_p(A,B)|\leq 2 p^{1/2}(p,A,B)^{1/2}.
\end{equation}
Our investigation of the case $m\geq 2$ is founded on work of Cochrane and Zheng \cite[\S 3]{cochrane-zheng}, with  
$f(x)=f_1(x)/f_2(x)$. Note that 
$$
f'(x)=\frac{2Ax^3-B}{x^2}.
$$
Following \cite[Eq.~(1.8)]{cochrane-zheng} and recalling that $p$ is odd, 
we put 
\begin{align*}
t=\ord_p(f')
&=\ord_p(2Ax^3-B)-\ord_p(x^2)\\
&=v_p\left((A,B)\right).
\end{align*}
Here, if $\ord_p(h)$ is the largest power of $p$ dividing all of the coefficients of a polynomial $h\in \ZZ[x]$, then $\ord_p(f_1/f_2)=\ord_p(f_1)-\ord_p(f_2)$.
Next,  we put 
$$
\mathcal{A}=\left\{\alpha\in \FF_p^*: 2A'\alpha^3 \equiv B' \imod{p}\right\},
$$
where $A'=p^{-t}A$ and $B'=p^{-t}B$.
In particular $(p,A',B')=1$ and $\#\mathcal{A}\leq 3$.
The elements of $\mathcal{A}$ are called the {\em critical points}.
If  $p\mid A'$ or $p \mid B'$ 
then $\mathcal{A}$ is empty since $(p,A',B')=1$. We therefore suppose that $p\nmid A'B'$. 

The strength of our estimate for $T_{p^m}(A,B)$ depends on the multiplicity $\nu_\alpha$ of each $\alpha\in \mathcal{A}$.
Suppose first that $p>3$ and 
 write $r(x)=2A'x^3-B'$.
Any root of multiplicity exceeding $1$ must also be a root of $r'(x)=6A'x^2$.
Hence  any $\alpha\in \mathcal{A}$ satisfies $\nu_\alpha=1$ if $p>3$.
When $p=3$ we have $2A'\alpha^3-B'=(2A'\alpha-B')^3$ in 
$\FF_3$ and so $\mathcal{A}$ contains a single element $\alpha$ of 
 multiplicity $\nu_\alpha=3$.

Next, as in \cite[\S 1]{cochrane-zheng}, one writes
$$
T_{p^m}(A,B)=\sum_{\alpha\in \FF_p^*} S_\alpha,
$$
with 
$$
S_\alpha=
\sumstar_{\substack{n\imod{p^m}\\ n\equiv {\alpha}\imod{p}}} e_{p^m}\left(\frac{f_1(n)}{f_2(n)}
\right).
$$
We are now ready to establish Lemma \ref{lem2}. 
When $m\leq t$ we have $T_{p^m}(A,B)=\phi(p^m)$, which is satisfactory. When $m=t+1$ we have 
\begin{align*}
T_{p^m}(A,B)
&=p^{m-1}T_{p}(A',B'),
\end{align*}
which has absolute value at most $2p^{m-1/2}\leq 2p^{m/2+t/2}$, 
by 
\eqref{eq:weil}.
It remains to deal with the case $m\geq t+2$.
Then \cite[Thm.~3.1(a)]{cochrane-zheng} implies that  
$
S_\alpha=0$ unless $\alpha\in \mathcal{A}$.
If $\alpha\in \mathcal{A}$  then this same result  
yields 
$$
|S_{\alpha}| \leq \nu_\alpha p^{t/(\nu_\alpha+1)}p^{m(1-1/(\nu_\alpha+1))}.
$$
We recall that $\#\mathcal{A}\leq 3$ and $\nu_\alpha=1$ for each $\alpha\in \mathcal{A}$ if $p>3$, while
$\#\mathcal{A}=1$ and $\nu_\alpha=3$ if $p=3$. 
Substituting this into our expression for $T_{p^m}(A,B)$, 
this therefore concludes the proof of the lemma.
\end{proof}

We complement our analysis of 
$T_{p^m}(A,B)$ for odd $p$ by studying the  exponential sum
\begin{equation}\label{eq:T2}
T_{2^m}(A,B;\delta)=
\sum_{\substack{n\imod{2^m} \\ 2\nmid n}} e_{2^{m+\delta}}\left(
An^2+2^\delta B\bar n\right),
\end{equation}
for $A,B\in \ZZ$ and
$\delta\in \{0,1\}$.
When $\delta=0$ we have  $T_{2^m}(A,B;0)=T_{2^m}(A,B)$, in our earlier notation. Furthermore, 
on writing $x=u+2^mv$ for $u\in (\ZZ/2^m\ZZ)^*$ and $v\in \ZZ/2\ZZ$, 
it is easy to check that 
$$
\sumstar_{x\imod{2^{m+1}}} e_{2^{m+1}} \left(\frac{Ax^3+2B}{x}\right)=
2
T_{2^m}(A,B;1).
$$
Hence we have 
$$
T_{2^m}(A,B;\delta)=
\frac{1}{2^\delta}
\hspace{0.2cm}
\sumstar_{x\imod{2^{m+\delta}}} e_{2^{m+\delta}} \left(\frac{Ax^3+2^\delta B}{x}\right),
$$
for $\delta\in \{0,1\}$, 
which brings our sum in line with the exponential sums considered by Cochrane and Zheng \cite{cochrane-zheng}.
We proceed to establish the following result.

\begin{lemma}\label{lem2'}
Let $\delta\in \{0,1\}$ and $m\in \NN$.
Then we have 
$$
|T_{2^m}(A,B;\delta)|
\leq 6 \cdot 
2^{3m/4}(2^m,A,B)^{1/4}.
$$
\end{lemma}

\begin{proof}
Let us put
$t=v_2\left((2A,2^\delta B)\right)$. Then $u\leq t\leq 1+u$, with 
$u=v_{2}((A,B))$.
Suppose first that  $m\leq  t+2$. Then the trivial bound gives
\begin{align*} 
|T_{2^m}(A,B;\delta)|\leq \phi(2^m)=2^{m-1}
&\leq  2^{(m+\min\{m,u\}+1)/2},
\end{align*}
which is satisfactory for the lemma.  We henceforth assume that 
$m\geq t+3$. 

We are interested in a complete exponential sum modulo $2^{m+\delta}$.
Arguing as in the proof Lemma \ref{lem2} we have 
$$
f'(x)=\frac{2Ax^3-2^\delta B}{x^2}
$$
and 
$\ord_2(f')=t$.
Next,  we put 
$$
\mathcal{A}=\left\{\alpha\in \FF_2^*: A'\alpha^3 \equiv B' \imod{2}\right\},
$$
where  $A'=2^{1-t}A$ and $B'=2^{\delta-t}B$.
In particular, $A',B'$ are  integers which cannot both be even and $\mathcal{A}$ consists of at most $1$ 
element and it has multiplicity at most $3$.
It therefore follows from   \cite[Thm.~3.1(b)]{cochrane-zheng} that 
$S_\alpha=0$ unless $\alpha\in \mathcal{A}$, in which case
$
|S_{\alpha}| \leq 3\cdot 2^{t/4+3(m+\delta)/4}.
$
But then 
$$
|T_{2^m}(A,B;\delta)|
\leq 
3\cdot 2^{(1+u)/4+3(m+\delta)/4}
= 6 \cdot 2^{3m/4+u/4}.
$$
This too is satisfactory for the lemma and so completes its proof.
\end{proof}

For any $q\in \NN$ we will henceforth write 
$q=q_0q_1$, where
\begin{equation}\label{eq:notation}
q_1=\prod_{\substack{p^j\| q \\ p>3}} p^j.
\end{equation}
That is, $q_0$ is not divisible by primes other than $2$ and $3$, while $q_1$ is coprime to $6$. 
Using the multiplicativity property \eqref{eq:mult}, we may  
combine Lemma \ref{lem2} and Lemma \ref{lem2'} with $\delta=0$ 
to   arrive at the following result. 

\begin{lemma}\label{lem3}
Let $q\in \NN$ and let $A,B\in \ZZ$.
Then we have 
$$
|T_{q}(A,B)|
\leq 18\cdot
3^{\omega(q_1)} 
q_0^{3/4}q_1^{1/2} (q_0,A,B)^{1/4}
(q_1,A,B)^{1/2},
$$
where $\omega(q_1)$ is the number of distinct prime factors of $q_1$.
\end{lemma}

\section{Error terms\label{sec:errorterms}}

The purpose of this section is to estimate $E(y)$ in 
\eqref{eq:errorterms}.
We begin with the case $c=0$. Then $d=\pm1$ by coprimality, and  \cite[Eq.~(25)]{strombergsson_effective_2013} yields
\[
\frac12 \int_{-1}^1 \tilde f_n\left(\pm \begin{pmatrix}\sqrt y& x/\sqrt y\\0&1/\sqrt y\end{pmatrix}\right)dx\ll \|f\|_{C_{\mathrm{b}}^2}\frac{y}{n^2}.
\]
After summing over $n$, the contribution from this  term is clearly much smaller than that claimed in Theorem~\ref{th:main}.

Next we consider the effect of 
shifting the interval of integration by $2$ in \eqref{eq:errorterms}. 
For this it will be convenient to note that 
\begin{align*}
\tilde f_n \left(\begin{pmatrix}a&b\\c&d\end{pmatrix} \begin{pmatrix}\sqrt y&(x-2)/\sqrt y\\0& 1/\sqrt y\end{pmatrix}\right) 
&=
\tilde f_n \left(\begin{pmatrix}a&b\\c&d\end{pmatrix} \begin{pmatrix} 1 & -2\\ 0& 1\end{pmatrix} \begin{pmatrix}\sqrt y&x/\sqrt y\\0& 1/\sqrt y\end{pmatrix}\right)\\
&=\tilde f_n \left(\begin{pmatrix}a&b-2a\\c&d-2c\end{pmatrix}  \begin{pmatrix}\sqrt y&x/\sqrt y\\0& 1/\sqrt y\end{pmatrix}\right)
\end{align*}	
and 
\begin{align*}
e\left(n\left( \frac{d(x-2)} 2 +\frac{c(x-2)^2}4\right)\right) &
=e\left(n\left( \frac{dx} 2 +\frac{cx^2}4 - cx\right)\right) \\
&= e\left(n\left( \frac{(d-2c)x} 2 +\frac{cx^2}4\right)\right) .
\end{align*}
Bearing these in mind it follows that 
for any $D\in\Z$ and $s\in\R$ we have
\begin{align*}
& \sum_{\substack{(c,d)=1\\ d\in[D,D+2c)}} \frac12 \int_{s}^{s+2} e\left(n\left( \frac{dx} 2 +\frac{cx^2}4\right)\right) \tilde f_n \left(\begin{pmatrix}a&b\\c&d\end{pmatrix} \begin{pmatrix}\sqrt y&x/\sqrt y\\0& 1/\sqrt y\end{pmatrix}\right) dx\\
&=\sum_{\substack{(c,d)=1\\ d\in[D,D+2c)}} \frac12 \int_{s+2}^{s+4} e\left(n\left( \frac{(d-2c)x} 2 +\frac{cx^2}4\right)\right) \tilde f_n \left(\begin{pmatrix}a&b-2a\\c&d-2c\end{pmatrix}  \begin{pmatrix}\sqrt y&x/\sqrt y\\0& 1/\sqrt y\end{pmatrix}\right) dx\\
&=\sum_{\substack{(c,d)=1\\ d\in[D-2c,D)}} \frac12 \int_{s+2}^{s+4} e\left(n\left( \frac{dx} 2 +\frac{cx^2}4\right)\right) \tilde f_n \left(\begin{pmatrix}a&b-2a\\c&d\end{pmatrix}  \begin{pmatrix}\sqrt y&x/\sqrt y\\0& 1/\sqrt y\end{pmatrix}\right) dx.
\end{align*}
But the  values of $a$ and $b$ are immaterial and so the contribution to  
\eqref{eq:errorterms} from terms with $c\neq 0$ is 
$$
\sum_{\substack{n\ge1\\ c\geq 1}}\frac12 
\sum_{\substack{(c,d)=1\\ d\imod {2c}}} \int_{\R}
e\left(n\left(\frac{dx}2+\frac{cx^2}{4}\right)\right)\tilde f_n\left(\begin{pmatrix}*&*\\c&d\end{pmatrix} \begin{pmatrix}\sqrt y&x/\sqrt y\\0&1/\sqrt y\end{pmatrix}\right)dx.
$$

Next,  
we change to Iwasawa coordinates as in \eqref{eq:iwasawa} (cf.\ \cite[Lemma~6.1]{strombergsson_effective_2013}). This leads to the expression 
\begin{equation}\label{eq:change_of_variables}
\begin{split}
\int_\R e&\left(n\left(\frac{dx}2+\frac{cx^2}4\right)\right) \tilde f_n \left(\begin{pmatrix}a&b\\c&d\end{pmatrix} \begin{pmatrix}\sqrt y&x/\sqrt y\\0& 1/\sqrt y\end{pmatrix}\right) dx \\
&= \int_0^\pi \tilde f_n\left(\frac{a}c -\frac{\sin 2\theta}{2c^2 y}, \frac{\sin^2\theta}{c^2y},\theta\right) e\left(-\frac{nd^2}{4c}+\frac{ncy^2\ctg^2 \theta}4\right)\frac{y\, d\theta}{\sin^2\theta},
\end{split}
\end{equation}
for positive $c$. For negative $c$, the limits on the integral are $-\pi$ and $0$. Since $ad-bc=1$ and $a$ and $b$ are otherwise arbitrary, we write $a=\bar d$ for any integer such that $\bar d d\equiv 1 \imod c$. Combining the integrals for positive and negative $c$ we get the contribution
\beq\sum_{\substack{ n\ge 1\\ c\ge 1}}\int_{-\pi}^\pi \sum_{\substack{ d\imod{2c}\\(c,d)=1}}\tilde f_n\left(\frac{\bar d}c -\frac{\sin 2\theta}{2c^2 y}, \frac{\sin^2\theta}{c^2y},\theta\right) e\left(-\frac{nd^2}{4c}+\frac{ncy^2\ctg^2 \theta}4\right)\frac{y\, d\theta}{\sin^2\theta}. \label{eq:fullrangetheta}\eeq

Recall that $\tilde f_n$ is left-invariant under $\left(\begin{smallmatrix} 1 & \Z\\0&1\end{smallmatrix}\right)$, which in Iwasawa coordinates translates into having period 1 in the first coordinate.
Therefore we can expand $\tilde f_n$ as a Fourier series to get 
\beq\label{eq:fseries}\tilde f_n\left(\frac{\bar d}c -\frac{\sin 2\theta}{2c^2 y}, \frac{\sin^2\theta}{c^2y},\theta\right)=\sum_{l\in\Z}b_l^{(n,c)}(\theta) e\left(\frac{l\bar  d}{c}\right)
 e\left(\frac{-l\sin 2\theta}{2c^2y}\right),
\eeq
whence the expression in \eqref{eq:fullrangetheta} is at most
\beq\label{eq:tobound}
\ll \sum_{n,c,l}\int_{-\pi}^\pi\Bigg\vert\sum_{\substack{d\imod{2c}\\ (c,d)=1}} e\left(-\frac{nd^2}{4c}+\frac{l\bar d}{c}\right)\Bigg \vert |b_l^{(n,c)}(\theta)| \frac{y\, d\theta}{\sin^2\theta}.
\eeq
We need bounds for the Fourier coefficients and the exponential sum in \eqref{eq:tobound}. 
Beginning with the former we have the following result.

\begin{lemma}
We have
\beq\label{eq:fourierbound} b_l^{(n,c)}(\theta)\ll \begin{cases}\|f\|_{C_{\mathrm{b}}^m}\min\left\{1, \left(\dfrac{|\sin\theta|}{nc\sqrt y}\right)^m\right\}& \text{for any }m\ge 0,\\
                               l^{-2}\|f\|_{C_{\mathrm{b}}^{m+2}} n^{-4}\min\left\{1, \left(\dfrac{|\sin\theta|}{nc\sqrt y}\right)^{m-4}\right\}& \text{for any }m\ge 4.
                              \end{cases}
\eeq
\end{lemma}

\begin{proof}
The first inequality follows from Lemma \ref{lem:andreas_fourier1} by taking the smaller of the estimate for general $m$ and $m=0$. To obtain the  second inequality we observe that 
\beq
\label{eq:bdef} b_l^{(n,c)}(\theta) = \int_{0}^1 \tilde f_n\left (u, \frac{\sin^2\theta}{c^2y},\theta\right) e(-lu) \, du.\eeq
We then apply integration by parts twice, followed by two applications of 
 Lemma \ref{lem:andreas_fourier2}, one with $k_1=2$, $k_2=k_3=0$, $m=4$, and the other with $k_1=2$, $k_2=k_3=0$, and $m$ general. Taking the smaller of the two outcomes yields the result.
\end{proof}

Next we turn to the exponential sum, with the following outcome.

\begin{lemma}\label{lem:expo}We have
$$
  \Bigg|\sum_{\substack{d\imod{2c}\\ (d,c)=1}} e\left(-\frac{nd^2}{4c}+\frac{l\bar d}{c}\right) \Bigg|\ll 
3^{\omega(c_1)} 
c_0^{3/4}c_1^{1/2}(c_0,n,l)^{1/4}
(c_1,n,l)^{1/2},
$$
where $c=c_0c_1$ and $c_1$ is given by \eqref{eq:notation}. 
\end{lemma}

\begin{proof}
Let $S(l,n;c)$ denote the  exponential sum in the statement of the lemma.
We need to relate $S(l,n;c)$ 
to the complete exponential sums considered in Section 
\ref{sec:exponential}.

The sum over $d$ runs modulo $2c$ in 
$S(l,n;c)$.
Let us write 
$$
S(l,n;c)=S_1+S_2,
$$
where $S_1$ is 
the contribution to the sum from even $d$  and $S_2$ is the remaining contribution. 
Writing $d=2d'$, we see that 
$$
S_1=\sum_{\substack{d'\imod{c}\\ (2d',c)=1}} e\left(-\frac{n{d'}^2}{c}+\frac{l\overline{2d'}}{c}\right)
=\begin{cases}
T_c(-n,\bar 2 l)
, &\mbox{if $2\nmid c$,}\\
0, &\mbox{if $2\mid c$.}
\end{cases}
$$
in the notation of \eqref{eq:Tq}.
The desired estimate for $S_1$ is now a direct consequence of  Lemma~\ref{lem3}.
Next, we note that 
$$
S_2=\sum_{\substack{d\imod{2c}\\ (d,2c)=1}} e_{2c}\left(- \frac{nd^2}{2}+2l\bar d\right),
$$
where $\bar d$ is now the multiplicative inverse of $d$ modulo $2c$.
If $n$ is even then 
$S_2=T_{2c}(-n/2,2l)$, which can again be estimated using 
 Lemma \ref{lem3}. If, on the other hand, $n$ is odd we write
 $c=2^{m-1}c'$ for odd $c'\in \NN$, 
Then  $S_2$ factorises as the product of an exponential sum modulo $2^{m}$ and an exponential sum modulo $c'$. A satisfactory estimate for the  latter follows from  Lemma 
\ref{lem3}.  Using \eqref{eq:mult}, 
the former
 is equal to
$T_{2^m}(-n\overline{c'}, 4l \overline{c'};1)$, in the notation of 
\eqref{eq:T2}, where $\overline{c'}\in \ZZ$ is chosen to satisfy 
$c'\overline{c'}\equiv 1\imod{2^{m+1}}$. In this case $(2^m,-n,4l)=1$ since $n$ is odd. 
This can be estimated using Lemma \ref{lem2'}, which ultimately leads to a 
satisfactory estimate for 
$|S_2|$.  This concludes the proof of the lemma.
\end{proof}

We learn from \cite[Eq.~(35)]{strombergsson_effective_2013} that 
\beq\label{eq:strominequality}
 \frac{1}{a(1+a)} \ll
\int_{-\pi}^\pi \min\left\{1, \left(\frac{|\sin\theta|}a\right)^2\right\} \frac{d\theta}{\sin^2\theta}
\ll  \frac{1}{a(1+a)},
\eeq
for $a>0$. 
We will apply this  with $a=nc\sqrt y$. 

Returning to \eqref{eq:tobound}, 
we recall the factorisation $c=c_0c_1$, where
$c_0$ is not divisible by primes greater than $3$, and where
$c_1$ given by \eqref{eq:notation}. 
It will be useful to note that 
\begin{equation}\label{eq:gamma}
\sum_{c_0}c_0^{-\gamma}=\sum_{\alpha,\beta\geq 0} 2^{-\alpha\gamma}3^{-\beta\gamma}=O_\gamma(1),
\end{equation}
for any $\gamma>0$.
We first consider  the case $l=0$. Combining the first line of \eqref{eq:fourierbound} with $m=2$ and 
Lemma~\ref{lem:expo}, we obtain the contribution 
\begin{align*}
&\ll
 \|f\|_{C_{\mathrm{b}}^2}
 \sum_{\substack{n\ge 1\\c\ge 1}} \frac{y}{nc\sqrt y(1+nc\sqrt y)} 3^{\omega(c_1)}c_0^{3/4}c_1^{1/2}(c_0,n)^{1/4} (c_1,n)^{1/2}\\
&\ll \|f\|_{C_{\mathrm{b}}^2}\sqrt y\sum_{n,c_0,c_1}\frac{3^{\omega(c_1) }(c_0,n)^{1/4}(c_1,n)^{1/2}}{nc_0^{1/4}
c_1^{1/2}(1+nc_0c_1\sqrt y)}\\
&\ll \|f\|_{C_{\mathrm{b}}^2}\sqrt y\sum_{n,c_1}\frac{3^{\omega(c_1) }(c_1,n)^{1/2}}{n^{3/4}
c_1^{1/2}(1+nc_1\sqrt y)},
\end{align*}
by \eqref{eq:gamma}.
The resulting sum is only made larger by summing over all positive integers $c$ and so we freely replace $c_1$ by $c$. 
Let us denote the right hand side by $J$. 
Writing $h=(c,n)$ and $c=hc'$ and $n=hn'$, we see that 
\begin{align*}
J
&\ll \|f\|_{C_{\mathrm{b}}^2}\sqrt y
\sum_{h,n',c'}\frac{3^{\omega(hc')}}{h^{3/4}{n'}^{3/4}{c'}^{1/2}(1+h^2n'c'\sqrt y)}.
\end{align*}
To proceed further we recall (see Tenenbaum \cite[Ex.~I.3.4]{tenenbaum_introduction_1995}, for example) 
that there is an absolute constant $C>0$ 
such that 
$$
\sum_{n\leq x}3^{\omega(n)}= C x\log^2 x +O(x\log x),
$$
for any $x\geq 2$. 
The bounds 
\beq\label{eq:omega_formulas}
\sum_{c>x} \frac{3^{\omega(c)}}{c^{3/2}}\ll \frac{\log^2 (2+x)}{x^{1/2}} 
\quad \mbox{and} \quad
\sum_{c\leq x} \frac{3^{\omega(c)}}{c^{1/2}}\ll x^{1/2}\log^2( 2+x)
\eeq
now follow from this using partial summation and 
 are valid for any $x>0$.
We therefore obtain
$$
\sum_{c'>(h^2n'\sqrt{y})^{-1}}\frac{3^{\omega(hc') }}{h^{3/4}{n'}^{3/4}{c'}^{1/2}(1+h^2n'c'\sqrt y)}
\ll \frac{3^{\omega(h)}\log^2 (2+(h^2n'\sqrt{y})^{-1}) }{h^{7/4}n'^{5/4}y^{1/4}},
$$
and similarly,
$$
\sum_{c'\leq(h^2n'\sqrt{y})^{-1}}\frac{3^{\omega(hc') }}{h^{3/4}{n'}^{3/4}{c'}^{1/2}(1+h^2n'c'\sqrt y)}
\ll \frac{3^{\omega(h)}\log^2 (2+(h^2n'\sqrt{y})^{-1}) }{h^{7/4}n'^{5/4}y^{1/4}}.
$$
It therefore follows that 
\begin{align*}
J
&\ll 
\|f\|_{C_{\mathrm{b}}^2}\sqrt y
\sum_{h,n'}
\frac{3^{\omega(h)}\log^2 (2+(h^2n'\sqrt{y})^{-1})
 }{h^{7/4}n'^{5/4}y^{1/4}}\\
&\ll \|f\|_{C_{\mathrm{b}}^2}y^{1/4}\log^2 (2+y^{-1}). 
\end{align*}

In the case $l\ne 0$, we apply the second inequality from \eqref{eq:fourierbound} with $m=6$, together with \eqref{eq:strominequality}. The contribution of these terms is therefore found to be 
\begin{align*}
 &\ll 
 \|f\|_{C_{\mathrm{b}}^8}
 \sum_{n,c,l}3^{\omega(c_1)} c_0^{3/4}c_1^{1/2}(c_0,n,l)^{1/4}(c_1,n,l)^{1/2} yl^{-2}  n^{-4}\frac{1}{nc\sqrt y(1+nc\sqrt y)}.
 \end{align*}
Using \eqref{eq:gamma} and  
taking $(c_0,n,l)^{1/4}\leq l^{1/4}$ and $(c_1,n,l)^{1/2}\leq l^{1/2}$, we see that this is  
\begin{align*}
&\ll 
\|f\|_{C_{\mathrm{b}}^8}
\sum_{n,c_0,c_1,l} 3^{\omega(c_1)} c_0^{3/4}c_1^{1/2} \frac{y}{l^{9/4} n^{4}} \frac{1}{nc_0c_1\sqrt y(1+nc_0c_1\sqrt y)}\\
&\ll 
\|f\|_{C_{\mathrm{b}}^8}
\sum_{n,c} 3^{\omega(c)} \frac{y^{1/2}}{n^{4} c^{1/2} (1+nc\sqrt y)}\\
&\ll
\|f\|_{C_{\mathrm{b}}^8}\left\{
 \sum_n \sum_{c>1/(n\sqrt y)} 3^{\omega (c)}  y^{1/2} \frac{1}{n^{5}c^{3/2}\sqrt y} +\sum_n\sum_{c\le 1/(n\sqrt y)} 3^{\omega(c)}  y^{1/2}\frac{1}{n^{4}c^{1/2}}\right\},
\end{align*}
as before.
Now we apply formulas \eqref{eq:omega_formulas}, which shows that the latter expression is at most 
\begin{align*}
&\ll\|f\|_{C_{\mathrm{b}}^8} \left\{\sum_n \frac1 {n^{5}} \frac{\log^2\left(2+(n\sqrt y)^{-1}\right)}{(n^{1/2}y^{1/4})^{-1}} +\sum_n \frac{y^{1/2}}{n^{4}} \frac{\log^2 \left(2+(n\sqrt y)^{-1}\right)}{n^{1/2}y^{1/4}}\right\}\\
&\ll \|f\|_{C_{\mathrm{b}}^8} y^{1/4}\log^2 (2+y^{-1}).
\end{align*}
This therefore concludes the  proof of Theorem \ref{th:main}.

\section{Proof of Corollary \ref{cor:sqrtn}\label{sec:sqrtn}}

We adopt the notation of \cite{ElkiesMcM04} for the most part.  
Define
\[
 \sigma_N(t) = \int_0^t \xi \,d\lambda_N(\xi).
\]
% It is this function that is realized on the homogeneous space $X$ in \cite{ElkiesMcM04}. 
For $c_+ > 0 > c_-$ let $\Delta_{c_-,c_+}\subset \R^2$ be the  open triangle bounded by the lines $w_1=1$, $w_2=2c_+ w_1$, $w_2=2c_- w_1$ in the $(w_1,w_2)$-plane; its area is clearly $c_+-c_-$. 
Also write $\Delta_c$ for $ \Delta_{0,c}$, $\emptyset$, or $\Delta_{c,0}$ according as $c$ is positive, zero, or negative. For a lattice translate $ \Gamma g\in\Gamma\quot G$, let 
\begin{align*}L(\Gamma g)=\sup_{c_+>0>c_-} \{c_+-c_-\colon \Delta_{c_-,c_+}\cap \Z^2 g = \emptyset\},\end{align*}
with the convention that $L(\Gamma g)=0$ if the set in the definition is empty and that $L(\Gamma g)=\infty$ if it is all of $\R^+.$ $L(\Gamma g)$ is the area of the largest triangle in the family $\Delta_{c_-,c_+}$ that is disjoint from $\Z^2 g$.

Following  \cite{ElkiesMcM04}, we establish a connection between homogeneous dynamics (embodied in the function $L$) and number theoretic quantities (embodied in $\sigma_N$ and $\lambda_N$). This is achieved in  Lemmas~\ref{lem:perfect_sq}, \ref{lem:homo}, \ref{lem:small_error}. For $y>0$, define probability measures $\mu_y$ on $X=\Gamma \quot G$ by 
 \begin{align*}
  \int_{X} f(\Gamma g) d\mu_y(g) = \frac12 \int_{-1}^1 f(\Gamma u(x)a(y)) \,dx. 
 \end{align*}
We also write $s=\lfloor N^{1/2}\rfloor $, and $\ind{\mathrm B}$  for the Boolean function
$$
\ind{\mathrm{B}}=\begin{cases}
1, & \text{if $\mathrm{B}=\mathrm{TRUE}$,}\\
0, & \text{if $\mathrm{B}=\mathrm{FALSE}$.}
\end{cases}
$$
The aim of this section is to prove the following result, of which Corollary \ref{cor:sqrtn} is a special case. 

\begin{prop}\label{prop:rate}
Let 
\begin{align}
 \sigma_\infty(t) & = \int_X  \ind{L(\Gamma g) < t} \, d\mu (g),\nonumber \\
 \lambda_\infty(t) & = \int_X \frac1{L(\Gamma g)} \ind{L(\Gamma g) < t} \, d\mu(g).\label{eq:lambda_inf}
\end{align}
Then for every $\eps>0$ and every $N$ we have 
\begin{align}
 \sigma_N(t)-\sigma_\infty(t) & \ll_\eps (1+t^2) N^{-\frac{1}{36}+\eps}, \label{eq:sigma_rate}\\
 \lambda_N(t) - \lambda_\infty(t) &\ll_\eps (1/t^2+t) N^{-\frac{1}{68}+\eps}.\label{eq:lambda_rate}
\end{align}

\end{prop}

Our first step in the proof of this result is the following estimate.

\begin{lemma}[Effective version of {\cite[Lemma~3.1]{ElkiesMcM04}}]\label{lem:perfect_sq}
 For $t>0$ and $s$ as above,
 we have
 \begin{align}
  \label{eq:lambda_connection}\lambda_N(t) & = \frac{s^2}N \lambda_{s^2} \left(ts^2/N\right) + O(N^{-1/2}),\\
  \sigma_N(t) & = \sigma_{s^2} \left( ts^2 / N\right) + O(tN^{-1/2}).\nonumber
 \end{align}
\end{lemma}

\begin{proof}
Consider \eqref{eq:lambda_connection}. We have 
\begin{align*}
 \lambda_{s^2}\left(ts^2/N\right) & = \frac{\#\{\text{gaps from }s^2\text{ points that are }<t/N\}}{s^2}\\
 & = \frac{\#\{\text{gaps from }N\text{ points that are }<t/N\}}{s^2} + O(N^{-1/2})\\
 & = \frac{N}{s^2} \frac{\#\{\text{gaps from }N\text{ points that are }<t/N\}}{N} + O(N^{-1/2})\\
 & = \frac{N}{s^2} \lambda_N(t) +O(N^{-1/2}).
\end{align*}
This proves the first statement. The second statement can be obtained by partial integration of the first, or directly via a similar argument. 
\end{proof}

% From now on assume that $N = s^2$ is a perfect square. 

For $N\geq 2$, 
let $L_N(\alpha )\colon \R/\Z\to [0,\infty)$ be $N$ times the length of the gap containing $\alpha $ (and $0$ if $\alpha \equiv \sqrt n \imod 1$ for some positive integer $n\in[1,N]$). Putting 
\[
 r_t(a,b) = \sqrt{a^2+b} - (a+t),
\]
we can write
\[
 L_N(t)  = N\left(\min_{r_t(a,b)\ge 0} r_t(a,b) - \max_{r_t(a,b) \le0} r_t(a,b)\right),
\]
where $a$ and $b$ range over integers such that $0<a<s$ and $0\le b\le 2a+1$. 
Let $I_N(t)$ denote the union of gaps that are less than $t/N$ in length. Equivalently, we put $I_N(t)=\{ \alpha \in\R/\Z\colon L_N(\alpha) <t\},$ and the sum of lengths of intervals comprising $I_N(t)$ equals $\sigma_N(t)$. 

A brilliant move of Elkies and McMullen was to replace the function $r_t(a,b)$ by another function, thereby moving the points of the sequence $\sqrt n \imod 1$ by a small amount and rendering the resulting point set amenable to techniques from homogeneous dynamics. Putting 
\[
 \tilde r_t(a,b) = \frac{a^2 + b - (a+t)^2}{2(a+t)},
\]
we write
\[
 \tilde L_N(t)  = N\left(\min_{\tilde r_t(a,b)\ge 0} \tilde r_t(a,b) - \max_{\tilde r_t(a,b) \le0} \tilde r_t(a,b)\right),
\]
with the same restrictions on $a$ and $b$. Let $\tilde I_N(t) = \{ \alpha \in\R/\Z\colon \tilde L_N(\alpha) <t\}$, and let $\tilde \sigma_N(t) $ denote the  combined length of segments comprising $\tilde I_N(t)$. Then we have

\begin{lemma}[Effective version of {\cite[Prop.~3.2 and Cor.~3.4]{ElkiesMcM04}}]\label{lem:homo}
 Let $\tilde \sigma_N(t)$ equal the length of the union of segments $I'_N(t)$. Then
 \begin{align*}
  \tilde \sigma_{s^2}((1-s^{-1/3})t) +O(s^{-1/3}) \le \sigma_{s^2}(t) \le \tilde \sigma_{s^2}((1+s^{-1/3})t) +O(s^{-1/3}). 
 \end{align*}
\end{lemma}

\begin{lemma}[cf.~{\cite[Prop.~3.8]{ElkiesMcM04}}]\label{lem:small_error} Let 
\begin{align*}
 \tilde {\tilde \sigma}_{s^2}(t) = \int_X  \ind{L(\Gamma g) < t} \, d\mu_{1/s^2}(g).
\end{align*}
Then we have
$ \tilde \sigma_{s^2}(t)  = \tilde {\tilde \sigma}_{s^2}(t) +O(s^{-1})$.
\end{lemma}

\begin{prop}\label{prop:smoothing}For each $\eps>0$, 
 \begin{align}\label{eq:sigma_smoothing}
  \int_X  \ind{L(\Gamma g) \ge t} \, d\mu_{1/s^2}(g) & = \int_X  \ind{L(\Gamma g) \ge t} \, d\mu(g) + O_\eps\left((1+t^2)N^{-\frac1{36}+\eps}\right)\\
   \int_X  \frac{\ind{L(\Gamma g) \ge t}}{L(\Gamma g) }  \, d\mu_{1/s^2}(g) & = \int_X  \frac{ \ind{L(\Gamma g) \ge t}}{L(\Gamma g) }  \, d\mu(g) + O_\eps\left((1/t^2+t)N^{-\frac1{68}+\eps}\right).\label{eq:lambda_smoothing}
 \end{align}

\end{prop}

\begin{proof}
To prove Proposition~\ref{prop:smoothing}, it suffices to apply Theorem \ref{th:main} to two bounded but glaringly discontinuous functions, 
\begin{align}
 \Gamma g & \mapsto  \ind{L(\Gamma g)\ge t}, \label{eq:function_sigma} \\
 \Gamma g & \mapsto \frac1{L(\Gamma g)} \ind{L(\Gamma g)\ge t}. \label{eq:function_lambda}
\end{align}
(Boundedness is assured by reversing the sense of the inequality in \eqref{eq:lambda_smoothing} versus \eqref{eq:lambda_inf}.)
We therefore approximate functions by smooth analogues first. To this end let $\delta\in(0,1/t)$; we will choose it later depending on $N$. 
Fix a left-invariant metric $d$ on $G$ coming from a Riemannian metric tensor, project it to $X$, and call the projected metric $d$, indulging the common abuse of notation. Also write $U_\delta$ for the $\delta$-neighbourhood of $1\in G$. 

%cite strom-venk
Following \cite{strombergsson_venkatesh_2005}, for each $\delta$ we fix $\psi_\delta$ to be a non-negative smooth compactly supported test function on $G$ so that
\begin{enumerate}
 \item $\int_G \psi_\delta(g) d\mu(g) = 1$;
 \item $\int_G |D\psi_\delta(g)  | d\mu(g) \ll_k \delta^{-k}$ for every monomial $D\in U(\mathfrak g)$ in the variables $X_1$, \dots, $X_5$ of order $k$;
 \item $\supp \psi_\delta \subset U_\delta$.
\end{enumerate}
For $f_1\colon G\to\R$ and $f_2\colon X \to \R$, define their convolution by 
\begin{align*}
 f_1 * f_2(\Gamma g) = \int_G f_1(h) f_2(\Gamma gh^{-1})d\mu(h).
\end{align*}

Let $f \colon X\to \R$ be one of the two functions \eqref{eq:function_sigma} or \eqref{eq:function_lambda},
and let $\Sing f$ be the subset of $X$ where $f$ is not smooth or where $L(\Gamma g)$ is infinite. The following smoothing technique will work for any function that is \emph{bounded}, has \emph{controllable} derivatives outside its singular set, and its singular set is \emph{thin}. 
For a subset $S$ of a metric space, write $\d_\delta S$ for the $\delta$-neighbourhood of the set $S$, where the metric on the ambient space is  implied. 
% We write $\d_\delta(S)$ for the $\delta$-neighbourhood of the set $S$, and the metric on the ambient space is  implied. 
Write
\begin{align*}
 f^\sharp_\delta (\Gamma g) & = 
 \begin{cases}
  \max f, & d(\Gamma g, \Sing f) < 3\delta,\\
    f(\Gamma g), & \text{otherwise},
 \end{cases}
\\
 f^\flat_\delta (\Gamma g) & = 
 \begin{cases}
   \min f, & d(\Gamma g, \Sing f) < 3\delta,\\
    f(\Gamma g), & \text{otherwise},
 \end{cases}
\end{align*}
for $f$ from \eqref{eq:function_sigma}. Let
\begin{align*}
 E_\delta = [0,\delta^{1/2}]\times [-1,1],
\end{align*}
and write 
\begin{align*}
 f^\sharp_\delta (\Gamma g) & = 
 \begin{cases}
  \max f, & d(\Gamma g, \Sing f) < 3\delta \text{ or }\Z^2 g \cap \d_{3\delta}E_\delta \ne \emptyset,\\
    f(\Gamma g), & \text{otherwise},
 \end{cases}
\\
 f^\flat_\delta (\Gamma g) & = 
 \begin{cases}
   \min f, & d(\Gamma g, \Sing f) < 3\delta \text{ or }\Z^2 g \cap \d_{3\delta}E_\delta \ne \emptyset,\\
    f(\Gamma g), & \text{otherwise},
 \end{cases}
\end{align*}
in the case of \eqref{eq:function_lambda}.  
For the two functions under consideration, maxima and minima are $1$ and $0$, and $1/t$ and $0$, respectively.

To construct approximating functions $f^\pm_\delta \colon X \to \R$ 
we need to understand smoothness properties of the function $L$. The following result is needed for establishing \eqref{eq:lambda_smoothing}. 

\begin{lemma}\label{lem:derivative_control}
 Suppose $g$ is such that $L(\Gamma g) \ne 0$, $L(\Gamma g) \ne \infty,$  and $\Z^2 g \cap E_\delta = \emptyset$. Assume also that $L$ is smooth at $\Gamma g$. Then we have
 \begin{align*}
  X_1.L(\Gamma g) & \ll 1,\\
    X_2.L(\Gamma g) & \ll L^2(\Gamma g),\\
  X_3.L(\Gamma g) & \ll L(\Gamma g),\\
  X_4.L(\Gamma g) & \ll L(\Gamma g) \delta^{-1/2}+L^2(\Gamma g),\\
  X_5.L(\Gamma g) & \ll \delta^{-1/2} + L(\Gamma g).
 \end{align*}
\end{lemma}

\begin{proof}
We only prove the statement for $X_4$; the others are similar. 

Since $L(\Gamma g ) \ne \infty$, it follows that $\Z^2 g \cap ((0,1)\times \R)$ is an infinite set. 
We let $(x,y)$  and $(x',y')$ be contained in this set so that they lie on the boundary of the triangle $\Delta_{y'/(2x'),y/(2x)}$ with $\Delta_{y'/(2x'),y/(2x)}$ disjoint from $\Z^2 g$ and such that $y>0>y'$. This implies that $L(\Gamma g) = \frac{y}{x} - \frac{y'}{x'}$. 
We assume first that there is only one such pair of points. 
Then
\begin{align*}
 X_4.L(\Gamma g) & = \frac{d}{d\eps} L(\Gamma g \exp \eps X_4)\bigg|_{\eps = 0}\\
 & = \frac{d}{d\eps} \left(\frac{y}{x+\eps}-\frac{y'}{x'+\eps}\right)\bigg|_{\eps=0}\\
 & = \frac{y}{x^2} - \frac{y'}{x'^2}.
\end{align*}
Since $(x,y), (x',y') \notin E_\delta$, we get the desired conclusion.

If there are several points on the boundary of $\Delta_{y'/(2x'),y/(2x)}$, we obtain the inequality 
\[
 X_4.L(\Gamma g) \le \max\left(\frac{\tilde y}{\tilde x^2} - \frac{\tilde y'}{\tilde x'^2}\right),
\]
where the maximum ranges over pairs 
\[
 (\tilde x,\tilde y) \in \d \Delta_{y'/(2x'),y/(2x)} \cap \Z^2 g \cap  ((0,1)\times \R_{>0}) ,
\]
and 
\[
 (\tilde x',\tilde y') \in \d \Delta_{y'/(2x'),y/(2x)} \cap \Z^2 g  \cap ((0,1)\times \R_{<0}) .
\]
Since $(\tilde x, \tilde y), (\tilde x', \tilde y') \notin E_\delta$ by hypothesis, we arrive that the same conclusion as with one pair of points.

\end{proof}

Let $C$ be a sufficiently large constant that depends implied constants in Lemma \ref{lem:derivative_control}. 
Then set
\begin{align*}
 f^+_\delta(\Gamma g ) & = (f^\sharp_\delta + C  \delta^{1/2}) * \psi_\delta(\Gamma g),\\
  f^-_\delta(\Gamma g ) & = (f^\flat_\delta - C \delta^{1/2}) * \psi_\delta(\Gamma g) 
\end{align*}
for $f$ from \eqref{eq:function_lambda}. For $f$ from \eqref{eq:function_sigma}, set
\begin{align*}
 f^+_\delta(\Gamma g ) & = f^\sharp_\delta  * \psi_\delta(\Gamma g),\\
  f^-_\delta(\Gamma g ) & = f^\flat_\delta  * \psi_\delta(\Gamma g) 
\end{align*}

\begin{lemma}\label{lem:norm}
For each $m= 0,1,2,\dots$ we have
$
 \|f^\pm_\delta\|_{C^m_\mathrm b} \ll \|f\|_{L^\infty} \delta^{-m}$.

\end{lemma}

\begin{proof}
Direct calculation.
\end{proof}

For $p\in \R^2$ let $X(p) = \{\Gamma g\in X : p\in \Z^2g\}$.

% \begin{lemma}\label{lem:approx_sigma}For integers $s$ and $N$ we have
% \begin{align} \sigma_N(t) & =  \frac 12 |\{x\in[-1,1]\colon L(u(x) a(1/N))\le t\}| + O((1+t)N^{-1/6})\\
%   \sigma_{s^2}(t) & =  \frac 12 |\{x\in[-1,1]\colon L(u(x) a(1/(s^2)))\le t\}| + O((s^2)^{-1/6}).
% \end{align}
% \end{lemma}

\begin{lemma}\label{lem:delta_nei}
Let $S$ be a subset of $\R^2$.  Then
\begin{align*}
 \d_\delta(X(p)) & \subset \bigcup_{d(p,p')\ll (1+|p|)\delta} X(p'),\\
 \d_\delta\left(\bigcup_{p\in S} X(p)\right) & \subset \bigcup_{d(p', S) \ll (1+|p'|)\delta} X(p'),
\end{align*}
where $d(\cdot,\cdot)$ denotes the Euclidean metric $\R^2$. 
\end{lemma}

\begin{proof}
Direct calculation. 
\end{proof}

We need some estimates for the measure of lattices that have nodes in a small set and nodes in another small set, as well as the measure of lattices that have nodes in a small set but no nodes in another large set. 
% 
% \begin{lemma}
% Let $p\in \R^2$. The space $X(p) = \{\Gamma g\in X : p\in \Z^2g\}$ can be identified with $\sltz\quot \sltr$ in a natural way, which is endowed with the Haar probability measure $\nu_p$. 
% \end{lemma}

\begin{lemma}
\label{lem:one_point}
Let $S\subset \R^2$  be a measurable set. Then $\mu\{\Gamma g \in X: \Z^2g \cap S \ne \emptyset\} \le \leb S$. 
\end{lemma}

\begin{proof}
 Follows from Markov's inequality. 
\end{proof}

\begin{lemma}\label{lem:affine_empty}Let $T$ be a measurable subset of $\R^2$. Then
$$\mu\{\Gamma g \in X: \Z^2g \cap T = \emptyset\} \le (1+\leb T)^{-1}.$$
\end{lemma}

\begin{proof}
See \cite[Theorem 1]{athreya_random_2014}. 
\end{proof}

% 
% \begin{lemma}Let $T$ be a bounded convex set in $\R^2$ that does not contain the origin, and let $\nu$ be the Haar probability measure on $\sltz\quot\sltr$. Then 
%  $\nu(\sltz g : \Z^2 g \cap T = \emptyset)\ll (1+\leb T)^{-1}.$
% \end{lemma}
% 
% \begin{proof}
% See \cite[Lemma 2.5]{strombergsson_probability_2011} or \cite[Theorem 2.2]{athreya_margulis_logarithm}.
% \end{proof}

\begin{lemma}\label{lem:two_sets}
Let $S_1$, $S_2$ be measurable subsets of $\R^2$. Then $$\mu\{\Gamma g: \Z^2g \cap S_1 \ne \emptyset, \Z^2g \cap S_2 \ne\emptyset\}\le \leb S_1 \leb S_2 + \leb(S_1 \cap S_2).$$
\end{lemma}

\begin{proof}
Follows from \cite[Propositions~7.10, 7.11]{marklof_strombergsson_free_path_length_2010}. 
\end{proof}

\begin{lemma}\label{lem:two_sets_empty}
Let $S_1$, $S_2$ be measurable subsets of $\R^2$. Then $$\mu\{\Gamma g: \Z^2g \cap S_1 \ne \emptyset, \Z^2g \cap S_2 = \emptyset\}\le \leb S_1 (1+\leb S_2)^{-1}.$$
\end{lemma}

\begin{proof}
Follows from \cite[Proposition~7.10]{marklof_strombergsson_free_path_length_2010} and   \cite[Theorem 2.2]{athreya_margulis_logarithm}. 
\end{proof}

% 
% \begin{lemma}[{\cite[Proposition~7.10]{marklof_strombergsson_free_path_length_2010}}]
% Let $p\in \R^2$, $X(p)$ as above, and let $\nu_p$ be the natural measure on $X(p)$ as in Section 7.2 of \cite{marklof_strombergsson_free_path_length_2010}. 
% \end{lemma}

We begin by studying the singular set of the function $\Gamma g \mapsto \ind{L(\Gamma g) \ge t}$, which is obviously $\d \{L(\Gamma g) \ge t\}.$ We analyse five possibilities: 
\begin{align*}
 \d \{L(\Gamma g) \ge t\} = \d & \{L(\Gamma g) \ge t\} \cap \{L(\Gamma g) = \infty\}\\
 &\bigcup \d  \{L(\Gamma g) \ge t\} \cap \{\infty> L(\Gamma g) > t\}\\
 &\bigcup \d  \{L(\Gamma g) \ge t\} \cap \{L(\Gamma g) = t\}\\
 &\bigcup \d  \{L(\Gamma g) \ge t\} \cap \{0< L(\Gamma g) < t\}\\
 &\bigcup \d  \{L(\Gamma g) \ge t\} \cap \{L(\Gamma g) =0\}.
\end{align*}

In the first case we have 
\begin{align*}
 \d_{2\delta}(\d  \{L(\Gamma g) \ge t\} \cap \{L(\Gamma g) = \infty\})& \subset \d_{2\delta}( \{L(\Gamma g) = \infty\})\\
 &\subset \d_{2\delta}\{\Z^2g \cap ((0,1)\times \R) = \emptyset\}\\
 &\subset \{\Z^2g \cap Q_\delta = \emptyset\},
\end{align*}
where $Q_\delta$ is the quadrilateral defined by inequalities 
\begin{align*}
 w_1& \gg (1+w_2)\delta\\
  w_1& \gg (1-w_2)\delta\\
 1-w_1& \gg (1+w_2)\delta\\
 1-w_1& \gg (1-w_2)\delta
\end{align*}
in the $(w_1,w_2)$-plane by Lemma \ref{lem:delta_nei}. Note that $L(\Gamma g) = \infty$ implies that the lattice has no nodes in the open strip $\{0<w_1<1\}$ so that (both) $c_+$, $c_-$ can be taken to be $\infty$ and $-\infty$, respectively. 
From Lemma \ref{lem:affine_empty}, we have that 
\begin{align*}
 \mu(\d_{2\delta}(\d  \{L(\Gamma g) \ge t\} \cap \{L(\Gamma g)= \infty\})) \ll 1/\leb(Q_\delta)\ll \delta. 
\end{align*}

In the second case we analyse 
\begin{align*}
 \d  \{L(\Gamma g) \ge t\} \cap \{\infty> L(\Gamma g) > t\}. 
\end{align*}
This set includes some lattices that contain the origin, as well as some lattices that contain a point on the line $\{w_1=1\}.$ In this first subcase we include \emph{all} lattices that contain the origin; their contribution is 
\begin{align*}
 \mu(\d_{2\delta}\{0\in \Z^2g\}) \ll \delta^2 
\end{align*}
by Lemma \ref{lem:delta_nei}. 

The second subcase comprises lattices 
that contain $(1,2h)$ from the segment  $\{1\}\times [-2t,2t]$  subject to the additional constraint that 
$\Delta_{h}$
does not meet the lattice. 
The $2\delta$-thickening of the set of such includes lattices with a node in 
$$
B_h^\delta =  [1-5(|h|+1)\delta,1+5(|h|+1)\delta] \times [h-5,h+5]
$$
and no node in 
\begin{equation}
 \Delta^\delta_h = \Delta_h \cap \{p\in\R^2 : d(p,\R^2\setminus \Delta_h)\ll 5(1+|h|)\delta\} \label{eq:thinning}
\end{equation}
for some integer $h$ between $-t-10$ and $t+10$. 
By an application of Lemma \ref{lem:two_sets_empty} with $S_1 = B_h^\delta$ and $S_2= \Delta_h^\delta$, the measure of a $2\delta$-thickening of this set is at most 
\begin{align}\label{eq:node_no_node}
\begin{aligned}
 \sum_{h=-\lfloor t\rfloor -10}^{\lceil t \rceil + 10} \frac{\leb B_h^\delta}{ 1+ \leb \Delta_h^\delta} & \ll \sum_{h=1}^{\lceil t \rceil +10} \frac{(h+1)\delta}{1+h}\\
 & \ll (t+1) \delta.
 \end{aligned}
\end{align}

% By an application of Lemma \ref{lem:two_sets_empty}, the measure of a $2\delta$-thickening of this set is at most a constant times 
% \begin{align*}
%  \int_0^{2t+1} \frac{\delta(1+h)}{1+h} dh \ll (t+1)\delta.
% \end{align*}

For the third case we use a combination of Lemma \ref{lem:delta_nei} and Lemma \ref{lem:two_sets}. Any lattice in $\{L(\Gamma g) = t\}$ will contain a node in $\Delta_t$; call this node $(v,2cv)$. If there are several such nodes, use one of the points with the smallest positive value of $c$. In addition to $(v,2cv)$, the lattice must contain a node on the (open) segment joining $(0,0)$ to $(1,2(t-c)v)$. If $t\le 10$, say, then we have 
\begin{align*}
 \mu(\d_{2\delta}\{L(\Gamma g) = t\}) \ll 100\times 100 \delta\ll \delta
\end{align*}
by Lemma \ref{lem:delta_nei}. So assume that $t$ is large, to wit $t>10$. Then $v$ is contained in $(0,1/t)\cup (1-1/t,1)$ by \cite[Lemma~3.12]{ElkiesMcM04}. Indeed, in order for the lattice to contain $(v,2cv)$ and have no nodes inside $\Delta_{c-t,c}$, the quantity in \cite[eq.~(3.44)]{ElkiesMcM04} must be positive, whence $v^2-v+1/(2t)>0$. Since $t$ is large, $v$ must be away from the axis of symmetry of this critical parabola; that is, in, say, $(0,1/t)\cup (1-1/t,1)$. Similarly, the node $(v',2c'v')$ of the lattice that lies on the open segment joining $(0,0)$ to $(1,2(t-c)v)$ must satisfy $v'\in (0,1/t)\cup (1-1/t,1)$. Thus, 
\begin{align*}
\mu(\d_{2\delta}\{L(\Gamma g) = t\}) \ll (t^2+1)\delta. 
\end{align*}

Case four does not arise. 

Case five is a singular case for lattices that meet the open segment $(0,1)\times \{0\}$. By Lemmas \ref{lem:delta_nei} and \ref{lem:one_point}, 
\begin{align*}
 \mu(\d_{2\delta}(\d  \{L(\Gamma g) \ge t\} \cap \{L(\Gamma g) =0\})) \ll \mu(\d_{2\delta}\{L(\Gamma g) =0\})\ll \delta. 
\end{align*}

We now have enough control to prove the first statement of Proposition \ref{prop:smoothing}; to prove the second statement it remains to understand the singular set of the function $\Gamma g\mapsto \frac1{L(\Gamma g)} \ind{L(\Gamma g)\ge t}$, which, considering our analysis of $\Gamma g\mapsto  \ind{L(\Gamma g)\ge t}$, amounts to studying $\Gamma g\mapsto \frac1{L(\Gamma g)}$. If the latter function is not smooth, then $\Z^2g$ contains a node in the set 
\begin{align*}
 ([0,1]\times \{0\} )
 \cup
 (\{1\}\times \R).
\end{align*}
We distinguish three cases. If this node is contained in $([0,1]\times \{0\} )
 \cup
 (\{1\}\times [-10,10])$, we use the bound 
 \begin{align*}
  \mu \big(\d_{2\delta}\{\Z^2 g \cap (([0,1]\times \{0\} )
 \cup
 (\{1\}\times [-10,10])) \ne\emptyset\}\big)\ll \delta.
 \end{align*}
The other two cases are controlled by $h\in(5,o(1/\delta))$ to be chosen later, and we need to recall that $\Gamma g$ where $L$ is not smooth are required not only to contain a node of the form $(1,2h')$, but also no nodes in the triangle $\Delta_{h'}$. 
(Without loss of generality assume $h'>0$.) 
When $h'>h$, we keep only the latter condition, meaning that 
\begin{equation}\begin{split}
\mu(\d_{2\delta}\{\Z^2g\cap \Delta_h =\emptyset\}) & \ll \mu(\Z^2g\cap \Delta^{\delta}_h =\emptyset\}\\
&\ll (1+\leb \Delta^{\delta}_h)^{-1}\\
&\ll (1+ h)^{-1}.\label{eq:h_large}
\end{split}
\end{equation}
Here $\Delta^{\delta}_h$ is a $2(1+|h|)\delta$-thinning of $\Delta_h$ as defined in \eqref{eq:thinning}. By virtue of the assumption $h<o(1/\delta)$, thinning barely modifies the original set, so that $\leb \Delta^{\delta}_h + 1 \gg \leb \Delta_h + 1= h+1$. The last line then follows by Lemma \ref{lem:affine_empty}. 

The final case is $5<h'\le h$. Here we use the method used to obtain \eqref{eq:node_no_node} above to get 
\begin{equation}\begin{split}
 \mu(\d_{2\delta}\{\Z^2g\cap \Delta_h =\emptyset, &\Z^2g \cap (\{1\}\times [10, h])\})\\
 & \ll \sum_{h' = 4}^{\lceil h \rceil + 1} \mu\{\Z^2g\cap \Delta_{h'}^\delta = \emptyset, \Z^2g\cap B_{h'}^\delta \ne \emptyset\}
 \\
 &
 \ll 
 \sum_{h' = 4}^{\lceil h \rceil + 1} \frac{\leb B_{h'}^\delta}{ 1+ \leb \Delta_{h'}^\delta}\\
 &
 \ll
 \sum_{h' = 4}^{\lceil h \rceil + 1} \frac{(h'+1)t}{h'+1} \\
 &
 \ll 
 (h+1)\delta.\label{eq:h_small}
\end{split}
\end{equation}
% 
% Here we resort to Lemma \ref{lem:two_sets_empty} to get the bound 
% \begin{align}
%  \int_{5}^{h+1} \frac{dh' \delta(1+h')}{1+h'}\ll (h+1)\delta. 
% \end{align}
The optimal choice for $h$ is $\delta^{-1/2}$, making \eqref{eq:h_large} and \eqref{eq:h_small} equal to $\delta^{1/2}.$ The measure of the thickened boundary is thus at most a constant times $(t^2+1)\delta$ or $(t^2+1)\delta + \delta^{1/2}$ for functions from \eqref{eq:function_sigma} or \eqref{eq:function_lambda}, accordingly.

We now prove that 
$$
 f^-_\delta \le f \le f^+_\delta
$$
everywhere on $X$. By symmetry, it is enough to establish the right hand inequality. Consider the case of the function from \eqref{eq:lambda_smoothing}. Suppose there exists $g$ so that $f(\Gamma g) > f^+_\delta(\Gamma g)$. Then there exists $g' \in gU_{\delta}$ such that $f(\Gamma g) > f^\sharp_\delta(\Gamma g') + C \delta^{1/2}.$ 
If $f^\sharp_\delta (\Gamma g') = f(\Gamma g'),$ then $d(\Gamma g', \Sing f)\ge 2\delta$ and $d(\Z^2 g', E_\delta)\ge 2\delta$ ($d$ denotes the Euclidean  metric on $\R^2$ in the second case). In particular, $\Gamma g$ and $\Gamma g'$ are contained in a ball of size $\delta$ that does not intersect $\Sing f$ or $\{\Gamma g'' : \Z^2 g'' \cap E_\delta \ne  \emptyset\},$ and $f^\sharp_\delta (\Gamma g') = f(\Gamma g')$ throughout this ball. By Lemma \ref{lem:derivative_control} and the Mean Value Theorem, $f(\Gamma g') -f(\Gamma g) $ cannot exceed a constant times 
\begin{align*}\left|\delta
D(1/L(\Gamma g''))\right| & \ll
\delta\left\lvert\frac{D.L(\Gamma g'')}{L^2(\Gamma g'')}\right\rvert\\
&\ll
\delta
\frac{1+L^2(\Gamma g'') + \delta^{-1/2}(1+L(\Gamma g''))}{L^2(\Gamma g'')}
\\
& \ll\delta^{1/2}\left(\frac1{t^2} + 1\right).\end{align*}
Implied constants here are absolute. Choosing $C$ to be a larger constant times $1+1/t^2$ will lead to a contradiction. 
If $f^\sharp_\delta (\Gamma g') = \max f$, then the contradiction is immediate. In the case of $f$ from \eqref{eq:sigma_smoothing}, the inequality is obvious. 

In order to prove the error terms in the statement of the theorem, we apply Theorem \ref{th:main} to $f^\pm_\delta$, which leads to the problem of optimising the expession 
$$
\int_{X}(f^+_\delta - f^-_\delta)d\mu + N^{-1/4+\eps}(\|f^+_\delta\|_{C^8_{\mathrm b}} + \|f^-_\delta\|_{C^8_{\mathrm b}})
$$
as a function of  $\delta$. 
For the first term in the case of \eqref{eq:lambda_smoothing} we have
\begin{align*}
 \int_{X}(f^+_\delta - f^-_\delta)\,d\mu 
 & = \int_X (f^\sharp_\delta -f^\flat_\delta + 2C\delta^{1/2})* \psi_\delta \, d\mu\\
 & = \int_X (f^\sharp_\delta -f^\flat_\delta)* \psi_\delta \, d\mu + O((1/t^2+1)\delta^{1/2})\\
 & \ll \int_{d(\Gamma g, \Sing f) < 4\delta \text{ or }\Z^2 g \cap \d_{3\delta} E_\delta \ne \emptyset}
\hspace{-1cm} (\max f -\min f)\, d\mu + O((1/t^2+1)\delta^{1/2})\\
& \ll \|f\|_{L^\infty} \mu(\d_{4\delta}(\Sing f)) +\|f\|_{L^\infty} \mu\{\Gamma g: \Z^2g \cap E_\delta \ne\emptyset\}\\ 
&\quad\quad+ O((1/t^2+1)\delta^{1/2})\\
& \ll \left(\frac{1+t^2}t+\frac1{t^2}+1\right)\delta^{1/2}\\
& \ll \left(\frac1{t^2}+t\right)\delta^{1/2}.
\end{align*}
In the case of \eqref{eq:sigma_smoothing} we have
\begin{align*}
 \int_{X}(f^+_\delta - f^-_\delta)\,d\mu  
  & = \int_X (f^\sharp_\delta -f^\flat_\delta)* \psi_\delta \, d\mu\\
  & \ll \mu(\d_{4\delta}(\Sing f))\\
  & \ll \delta (1+t^2).
\end{align*}
For \eqref{eq:sigma_smoothing},  we minimise $\delta(t^2+1) + N^{-1/4+\eps}\delta^{-8}$. On the other hand, for \eqref{eq:lambda_smoothing}, we instead minimise $\delta^{1/2} (1/t^2+t) + N^{-1/4+\eps}\delta^{-8}$. In each case we view $\delta$ as a function of $N$ only. The optimal choices are $\delta = N^{-\frac1{36}}$ and $\delta = N^{-\frac1{34}}$, respectively. This proves \eqref{eq:lambda_rate} and \eqref{eq:sigma_rate}.
\end{proof}

\begin{lemma}\label{lem:sigma_infinity_property}
For $\lambda_\infty(t)$ and  $\sigma_\infty(t)$ defined as in Proposition~\ref{prop:rate}, we have 
 \begin{align*}
  \frac{d}{dt} \sigma_\infty(t) & \ll \min(t, t^{-2}),\\ 
    \frac{d}{dt} \lambda_\infty(t) & \ll \min(1, t^{-3}),
 \end{align*}
\end{lemma}

\begin{proof}
 Follows from the explicit formulas for $\sigma_\infty$ and $\lambda_\infty$ from \cite[Prop.~3.14]{ElkiesMcM04}. 
\end{proof}

\begin{proof}[Proof of Proposition~\ref{prop:rate}]
We wish to establish control on $\sigma_N$ first, so we write
\begin{align*}
 \sigma_N(t) - \sigma_\infty(t) & \le \sigma_N(t) - \sigma_{s^2} (ts^2/N)
 \\&\phantom{=\,} + \sigma_{s^2} (ts^2/N) - {\tilde \sigma}_{s^2} ((1-s^{-1/3}) ts^2/N) 
 \\&\phantom{=\,} + {\tilde \sigma}_{s^2} ((1-s^{-1/3}) ts^2/N) - \tilde {\tilde \sigma}_{s^2} ((1-s^{-1/3}) ts^2/N) 
 \\&\phantom{=\,} + \tilde {\tilde \sigma}_{s^2} ((1-s^{-1/3}) ts^2/N) -\sigma_\infty((1-s^{-1/3}) ts^2/N) 
 \\&\phantom{=\,} +\sigma_\infty((1-s^{-1/3}) ts^2/N) -\sigma_\infty(t).
\end{align*}
Now using Lemmas~\ref{lem:perfect_sq}, \ref{lem:homo}, \ref{lem:small_error}, Proposition~\ref{prop:smoothing}, and Lemma~\ref{lem:sigma_infinity_property}, we arrive at the upper bound 
\begin{align*}
 O(tN^{-1/2}) +O(s^{-1/3})+O(s^{-1})+O_\eps((1+t^2)N^{-\frac1{36}+\eps})+O(s^{-1/3}).
\end{align*}
For the lower bound on the difference, we write 
\begin{align*}
 \sigma_N(t) - \sigma_\infty(t) & \ge \sigma_N(t) - \sigma_{s^2} (ts^2/N)
 \\&\phantom{=\,} + \sigma_{s^2} (ts^2/N) - {\tilde \sigma}_{s^2} ((1+s^{-1/3}) ts^2/N) 
 \\&\phantom{=\,} + {\tilde \sigma}_{s^2} ((1+s^{-1/3}) ts^2/N) - \tilde {\tilde \sigma}_{s^2} ((1+s^{-1/3}) ts^2/N) 
 \\&\phantom{=\,} + \tilde {\tilde \sigma}_{s^2} ((1+s^{-1/3}) ts^2/N) -\sigma_\infty((1+s^{-1/3}) ts^2/N) 
 \\&\phantom{=\,} +\sigma_\infty((1+s^{-1/3}) ts^2/N) -\sigma_\infty(t),
\end{align*}
which matches the upper bound upon application of Lemmas~\ref{lem:perfect_sq}, \ref{lem:homo}, \ref{lem:small_error}, Proposition~\ref{prop:smoothing}, and Lemma~\ref{lem:sigma_infinity_property}. The proof of \eqref{eq:sigma_rate} is thus complete. 

To get at $\lambda_N(t)$ we start with the following bounds. 
\begin{align}\label{eq:lambda_easy}\begin{split}
 \lambda_N(t) - \lambda_\infty(t) & = \lambda_N(t) - \tfrac{s^2}{N}\lambda_{s^2}(ts^2/N) \\
&\phantom{=\,}+  \tfrac{s^2}{N}(\lambda_{s^2}(ts^2/N)-1) - (\lambda_\infty(t)-1) + O(N^{-1/2}).
\end{split}
\end{align}
The first line is controlled by Lemma \ref{lem:perfect_sq}, while for the second we need a more complicated argument.  Write the second line without the error term as 
\begin{equation}\label{eq:lambda_hard}
\begin{split}
  \frac{s^2}{N}(\lambda_{s^2}(ts^2/N)-1) - (\lambda_\infty(t)-1) 
  = ~& \int_{t}^\infty \frac{d\sigma_{\infty}(\xi)}{\xi} - \frac{s^2}N \int_{ts^2/N}^\infty \frac{d\sigma_{s^2}(\xi)}{\xi}\\
  = ~& \int_{t}^\infty \frac{d\sigma_{\infty}(\xi)}{\xi}\\
  &
   - \frac{s^2}N \left( \int_{ts^2/N}^\infty \frac{\sigma_{s^2}(\xi)\,d\xi}{\xi^2}
    - \frac{\sigma_{s^2}(ts^2/N)}{ts^2/N}\right).
\end{split}
\end{equation}
The terms containing $\sigma_{s^2}$ are controlled above using
\begin{align*}
 \sigma_{s^2} (\xi) & = \sigma_\infty(\xi) \\
 & \phantom{=\,} + \sigma_{s^2}(\xi) - \tilde \sigma_{s^2}((1-s^{-1/3}) \xi) \\
 & \phantom{=\,} + \tilde \sigma_{s^2}((1-s^{-1/3}) \xi) - \tilde{\tilde \sigma}_{s^2}((1-s^{-1/3}) \xi) \\
 & \phantom{=\,} + \tilde{\tilde \sigma}_{s^2}((1-s^{-1/3}) \xi) -\sigma_\infty((1-s^{-1/3})\xi)\\
 & \phantom{=\,} + \sigma_\infty((1-s^{-1/3})\xi) - \sigma_\infty(\xi)
\end{align*}
and below using 
\begin{align*}
 \sigma_{s^2} (\xi) & = \sigma_\infty(\xi) \\
 & \phantom{=\,} + \sigma_{s^2}(\xi) - \tilde \sigma_{s^2}((1+s^{-1/3}) \xi) \\
 & \phantom{=\,} + \tilde \sigma_{s^2}((1+s^{-1/3}) \xi) - \tilde{\tilde \sigma}_{s^2}((1+s^{-1/3}) \xi) \\
 & \phantom{=\,} + \tilde{\tilde \sigma}_{s^2}((1+s^{-1/3}) \xi) -\sigma_\infty((1+s^{-1/3})\xi)\\
 & \phantom{=\,} + \sigma_\infty((1+s^{-1/3})\xi) - \sigma_\infty(\xi)
\end{align*}
Using Lemmas~\ref{lem:homo}, \ref{lem:small_error}, Proposition~\ref{prop:smoothing}, and Lemma \ref{lem:sigma_infinity_property} and the splittings above, we get
\begin{align*}
 \frac{s^2}N \bigg( \int_{\frac{ts^2}N}^\infty \frac{\sigma_{s^2}(\xi)\,d\xi}{\xi^2} - \frac{\sigma_{s^2}(
 \frac{ts^2}N)}{ts^2/N}\bigg) & 
 \le \frac{s^2}N \int_{ts^2/N}^\infty \frac{d\sigma_\infty(\xi)}{\xi} \\
 & \phantom{=\,}+
 \int_{ts^2/N}^\infty \frac{O(s^{-1/3}) \, d\xi}{\xi^2} - \frac{O(s^{-1/3})}{ts^2/N} \\
 &\phantom{=\,} + \int_{ts^2/N}^\infty \frac{O(N^{-1/2})\, d\xi}{\xi^2} - \frac{O(N^{-1/2})}{ts^2/N}\\
 &\phantom{=\,} + \left((\tfrac{N}{ts^2})^2+\tfrac{ts^2}{N}\right) O_\eps\left(N^{-\frac1{68}+\eps}\right) \\
 &\phantom{=\,} 
 - \frac{\left(1+(\tfrac{ts^2}{N})^2\right)O_\eps\left(N^{-\frac1{36}+\eps}\right)}{ts^2/N}
 + O(s^{-1/3}).
\end{align*}
The lower bound is the same except for implied constants. Therefore we can rewrite \eqref{eq:lambda_hard} as 
\begin{align*}
\frac{s^2}{N}(\lambda_{s^2}(ts^2/N)-1) - (\lambda_\infty(t)-1) & = \lambda_\infty(t) - \frac{s^2}N \lambda_\infty(ts^2/N) + O_\eps\left((1/t^2+t)N^{-\frac1{68}+\eps}\right)\\
& \ll_\eps (1/t^2+t)N^{-\frac1{68}+\eps}
\end{align*}
by Lemma~\ref{lem:sigma_infinity_property}. Substituting this result into \eqref{eq:lambda_easy}, we finally arrive at the required bound. 
\end{proof}

\section{Sketch of the proof of Theorem \ref{th:main_second}\label{sec:generaldensity}}
The proof of Theorem \ref{th:main_second} proceeds analogously to that of Theorem \ref{th:main}. The main term in equation \eqref{eq:main_term} becomes 
\[
 \int_{\R}\tilde f_0\begin{pmatrix}\sqrt y&x/\sqrt y\\0&1/\sqrt y\end{pmatrix} \rho(x) \, dx.
\]
From \cite[Eqs.~(23), (24)]{strombergsson_effective_2013} this equals 
\[\int_X f\,  d\mu \int_\R \rho(x) \, dx + O(\|f\|_{C_{\mathrm{b}}^4} \|\rho\|_{W^{1,1}} y^{1/2} \log^3 (2+1/y)),\]
which is satisfactory for the theorem. 

Turning to the new error terms, we treat the term $c=0$ following \cite[Eq.~(25)]{strombergsson_effective_2013}. Our 
analogue of the remaining contribution to the error term  $E(y)$ in \eqref{eq:errorterms} is
$$
E_\rho(y)=
\sum_{\substack{c,n\ge1\\ (c,d) =1}}
 \int_{\R} e\left(n\left(\frac{dx}2+\frac{cx^2}{4}\right)\right)\tilde f_n\left(\begin{pmatrix}*&*\\c&d\end{pmatrix} \begin{pmatrix}\sqrt y&x/\sqrt y\\0&1/\sqrt y\end{pmatrix}\right)\rho(x)dx.
$$
Now we proceed directly to the change of variables \eqref{eq:change_of_variables},  following \cite[Lemma 6.1]{strombergsson_effective_2013}. This gives 
\begin{align*}
\int_\R e\left(n\left(\frac{dx}2+\frac{cx^2}4\right)\right) \tilde f_n \left(\begin{pmatrix}a&b\\c&d\end{pmatrix} \begin{pmatrix}\sqrt y&x/\sqrt y\\0& 1/\sqrt y\end{pmatrix}\right) \rho(x)\, dx = \int_0^\pi g(\theta) d\theta,
\end{align*}
for $c>0$, where
$$
g(\theta)=\tilde f_n\left(\frac{a}c -\frac{\sin 2\theta}{2c^2 y}, \frac{\sin^2\theta}{c^2y},\theta\right) e\left(-\frac{nd^2}{4c}+\frac{ncy^2\ctg^2 \theta}4\right)
\rho\left(-\frac dc + y\ctg\theta\right)
\frac{y}{\sin^2\theta}.
$$
We have the same  integral with limits $-\pi$ and 0 if $c<0$. Combining terms with positive and negative $c$ gives
$$
E_\rho(y)=
\sum_{\substack{ c,n\ge 1}} \int_{-\pi}^\pi \sum_{(c,d)=1} g(\theta)d\theta.
$$
We periodise $\rho$ with period $2$, by setting  $P(z)=\sum_{m\in \Z} \rho(z+2m).$ Now the latter expression 
can be rewritten using only periodic functions as 
\begin{align*}
E_\rho(y)=
\sum_{\substack {c,n\ge 1 }} \int_{-\pi}^\pi \sum_{\substack{d\imod {2c}\\(c,d)=1}} &\tilde f_n\left(\frac{\bar d}c -\frac{\sin 2\theta}{2c^2 y}, \frac{\sin^2\theta}{c^2y},\theta\right) e\left(-\frac{nd^2}{4c}+\frac{ncy^2\ctg^2 \theta}4\right)\\
&\times
P\left(-\frac dc + y\ctg\theta\right)
\frac{y\, d\theta}{\sin^2\theta}.\end{align*}
Exploiting periodicity of $\tilde f_n$ and $P$, we replace them by their Fourier series leaving an exponential sum and two Fourier coefficients to control. Thus  
\beq\label{eq:new_final}
E_\rho(y)
\ll \sum_{\substack {c,n\ge 1 }} \int_{-\pi}^\pi 
\sum_{k,l\in\Z} \Bigg|\sum_{\substack{d\imod {2c}\\(c,d)=1}}e\left(-\frac{nd^2}{4c}+\frac{l\bar d}{c}-\frac{kd}{2c}\right)\Bigg| |b_l^{(n,c)}(\theta) \,a_k| \frac{y\,d\theta}{\sin^2 \theta},
\eeq
where $a_k$ are the Fourier coefficients of $P$ and $b_l^{(n,c)}(\theta)$ are as in \eqref{eq:fseries} and \eqref{eq:bdef}. 
In particular, we have 
\beq
\label{eq:measure_fourier}
a_k\ll_\eta (1+|k|)^{-1-\eta} \|\rho\|_{W^{1,1}}^{1-\eta} \|\rho\|_{W^{2,1}}^\eta
\eeq
for all $\eta\in (0,1)$, 
as can be seen  using integration by parts.

The exponential sum in \eqref{eq:new_final} can be estimated using the tools developed in Section \ref{sec:exponential}.
Note that the case $k=0$ corresponds precisely to the sum considered in 
Lemma \ref{lem:expo}.
We say $u\in\N$ is \emph{square-free} if $p\mid u$ implies $p^2\nmid u$ for every prime $p$, and similarly $v\in \NN$ is \emph{square-full} if $p\mid v$ implies $p^2\mid v$.

\begin{lemma}\label{lem:new_expo}
Write $c=c_0c_1=c_0uv$ with $u$ square-free, $v$ square-full, and furthermore $(uv,6)=(u,v)=1$. Then there exists an absolute  constant $K\in \NN$ such that 
\[
\left|
\sum_{\substack{d\imod {2c}\\ (c,d)=1}}e\left(
-\frac{nd^2}{4c}+
\frac{l\bar d}{c}-\frac{kd}{2c}\right)\right|
\le K^{\omega(c)} c_0^{3/4}u^{1/2}v^{2/3}
(c_0,k,n,l)^{1/4}
(u,k,n,l)^{1/2} (v,k,n,l)^{1/3}.
\]
\end{lemma}

\begin{proof}
We will sketch the proof of this result based on the methods of Section 
\ref{sec:exponential}.  In doing so we will not pay heed to the particular value of $K$.
Arguing as in the proof of 
Lemma \ref{lem:expo}, the main task is to estimate the exponential sum 
$$
T_{p^m}(A,B,C)=
\sumstar_{n\imod{p^m}} e_{p^m}\left(
An^2+B\bar n+Cn\right),
$$
for $A,B,C\in \ZZ$ and a prime power $p^m$.  
Note that $
T_{p^m}(A,B,0)=T_{p^m}(A,B)$ in the notation of \eqref{eq:Tq}.
We may now write $T_{p^m}(A,B,C)$ in the form 
\eqref{eq:fall},
with  $f_1(x)=Ax^3+B+Cx^2$ and $f_2(x)=x$.
When 
 $m=1$ it follows from  Bombieri \cite{bombieri} that 
$$
|T_p(A,B,C)|\leq 2 p^{1/2}(p,A,B,C)^{1/2}.
$$

When $m\geq 2$, we apply  Cochrane and and Zheng \cite{cochrane-zheng}, as before.
We see that
$$
f'(x)=\frac{2Ax^3-B+Cx^2}{x^2},
$$
whence $t=\ord_p(f')=v_p\left((2A,B,C)\right)$, 
 in the notation of  \cite[Eq.~(1.8)]{cochrane-zheng}.
This time we have 
$$
\mathcal{A}=\left\{\alpha\in \FF_p^*: A'\alpha^3 -B'+C'\alpha^2\equiv 0 \bmod{p}\right\},
$$
where $A'=p^{-t}2A$, $B'=p^{-t}B$ and $C'=p^{-t}C$.
In particular $(p,A',B',C')=1$.
We may henceforth assume that $m\geq t+3$ since the desired conclusion follows from the trivial bound otherwise. 
One finds that any critical point $\alpha\in \mathcal{A}$ has multiplicity $\nu_\alpha\leq 2$ when $p\neq 3$ and multiplicity $\nu_\alpha\leq 3$ when $p=3$.
Next, one applies 
\cite[Thm.~3.1]{cochrane-zheng} 
to deduce that 
\begin{align*}
T_{p^m}(A,B,C) &\ll 
\begin{cases}
p^{2m/3+\min\{m,t\}/3}, &\mbox{if $p\neq 3$,}\\
3^{3m/4+\min\{m,t\}/4}, &\mbox{if $p= 3$,}
\end{cases}
\end{align*}
for $m\geq 2$.

Once combined with our treatment of the case $m=1$, one 
 arrives at the statement of the lemma on invoking multiplicativity.
 \end{proof}

We have four regimes in \eqref{eq:new_final} to consider, according to whether or not $k$ or $l$ vanish. 
 The cases with $k=0$ are identical to those already dealt with in Section \ref{sec:errorterms}. We proceed
to  present the argument needed to handle  the case $l=0$ and  $k\ne 0$.
After using \eqref{eq:fourierbound} with $m=2$ in \eqref{eq:new_final},  followed by \eqref{eq:strominequality},  
Lemma \ref{lem:new_expo} and \eqref{eq:measure_fourier}, we arrive at the bound
\beq\label{eq:intermediate}
\begin{split}
E_\rho(y)
\ll~& 
 \|f\|_{C_{\mathrm{b}}^2}\|\rho\|_{W^{1,1}}^{1-\eta} \|\rho\|_{W^{2,1}}^\eta 
\\
&\times \sum_{c=c_0uv, n\geq 1} \sum_{k\ne 0}\frac{y}{|k|^{1+\eta}}\cdot \frac{K^{\omega(c)}c_0^{3/4}u^{1/2}v^{2/3}
 (c_0,k,n)^{1/4}
 (u,k,n)^{1/2} (v,k,n)^{1/3}  }{nc\sqrt y (1+ nc\sqrt y)}\\
=~& \|f\|_{C_{\mathrm{b}}^2}\|\rho\|_{W^{1,1}}^{1-\eta} \|\rho\|_{W^{2,1}}^\eta 
 S(y),
\end{split}
\eeq
say.
We will apply the upper bounds 
$(c_0,k,n)^{1/4}\leq c_0^{1/4-\eta/4}|k|^{\eta/4}$,
$(u,k,n)^{1/2}\leq (u,k)^{1/2}$, and 
$(v,k,n)^{1/3}\leq v^{1/3-\eta/4}|k|^{\eta/4}$ in order to simplify this expression.
In particular the resulting sum over $c_0$ is absolutely convergent by \eqref{eq:gamma}. 
Next  we    divide the sum so that $uv$ belongs to the dyadic intervals $[2^{j-1},2^{j})$ for $j\in \N$.
In this way we deduce that 
\begin{align*}
S(y)
&\ll \sum_{n\ge 1}
\sum_{k\ne 0}\frac{y^{1/2}}{|k|^{1+\eta/2}}
\sum_{j\ge 1} {\sum_{\substack{v\le 2^j \\ v\text{ sq.-full}}}} \sum_{2^{j-1}/v\leq u<  2^j/v}
\frac{K^{\omega(uv)}(u,k)^{1/2}v^{1/2-\eta/4} }{n\sqrt{uv}(1+nuv\sqrt y)}\\
&\ll  
 \sum_{n\ge 1} \frac{1}n 
\sum_{k\ne 0}\frac{y^{1/2}}{|k|^{1+\eta/2}}
  \sum_{j\ge1} \frac{1}{2^{j/2}(1+n\sqrt y2^j)}
{\sum_{\substack{v\le 2^j \\ v\text{ sq.-full}}}} K^{\omega(v)}v^{1/2-\eta/4}
\smashoperator{\sum_{u\le  2^j/v}} K^{\omega(u)}
(u,k)^{1/2}.
\end{align*}
Now, we have $ \sum_{n\le x} K^{\omega(n)} \ll x\log^{K-1} x$
for any $K>1$ and $x\geq 2$ (see \cite[Thm.~II.6.1]{tenenbaum_introduction_1995}, for example).
Hence it follows that 
\begin{align*}
\sum_{n\le x} K^{\omega(n)}(n,k)^{1/2} 
&\leq \sum_{h\mid k}h^{1/2} \sum_{\substack{n\leq x\\h\mid n}}K^{\omega(n)}\\
&\ll x\log^{K-1} x \sum_{h\mid k}h^{-1/2}K^{\omega(h)},
\end{align*}
for $x\geq 2$.
The remaining sum over $h$ is at most $\ll _K \tau(k)$,  where $\tau$ denotes the divisor function, 
whence
$$S(y)\ll y^{1/2} 
\sum_{n\ge 1} \frac{1}n 
\sum_{k\ne 0}\frac{\tau(k)}{|k|^{1+\eta/2}}
  \sum_{j\ge1} \frac{2^{j/2}\log^{K-1} 2^j}{1+n\sqrt y2^j}
{\sum_{\substack{v\le 2^j \\ v\text{ sq.-full}}}} \frac{K^{\omega(v)}}{v^{1/2+\eta/4}}.
$$
The sum over $k$ is convergent. Furthermore, 
since square-full integers have square root density, the sum over $v$ is also convergent. 
Hence
$$S(y)\ll y^{1/2} 
\sum_{n\ge 1} \frac{1}n 
  \sum_{j\ge1} \frac{2^{j/2}\log^{K-1} 2^j}{1+n\sqrt y2^j}
$$
Once substituted into \eqref{eq:intermediate},
this leads to the  satisfactory
contribution 
$$
\ll 
\|f\|_{C_{\mathrm{b}}^2}\|\rho\|_{W^{1,1}}^{1-\eta} \|\rho\|_{W^{2,1}}^\eta y^{1/4} \log^{K-1}(2+y^{-1}),
$$
to $E_\rho(y)$, for any $\eta\in (0,1)$.

In a similar manner, using 
instead the second estimate from \eqref{eq:fourierbound} with $m=6$, 
the contribution from  $kl\ne0$  is found to be 
$$
\ll \|f\|_{C_{\mathrm{b}}^8}\|\rho\|_{W^{1,1}}^{1-\eta} \|\rho\|_{W^{2,1}}^\eta y^{1/4} \log^{K-1}(2+y^{-1}),
$$
for any $\eta\in (0,1)$.
This completes the sketch of the proof of Theorem \ref{th:main_second}.

\bibliographystyle{plain}
\bibliography{bibliography}

\def\cprime{$'$}
\begin{thebibliography}{10}

\bibitem{athreya_random_2014}
Jayadev~S. Athreya.
\newblock Random {Affine} {Lattices}.
\newblock {\em arXiv:1409.6355, to appear in Contemporary Mathematics},
  September 2014.

\bibitem{athreya_margulis_logarithm}
Jayadev~S. Athreya and Gregory~A. Margulis.
\newblock Logarithm laws for unipotent flows. {I}.
\newblock {\em J. Mod. Dyn.}, 3(3):359--378, 2009.

\bibitem{bombieri}
Enrico Bombieri.
\newblock On exponential sums in finite fields.
\newblock {\em Amer. J. Math.}, 88:71--105, 1966.

\bibitem{burger_horocycle_1990}
Marc Burger.
\newblock Horocycle flow on geometrically finite surfaces.
\newblock {\em Duke Math. J.}, 61(3):779--803, 1990.

\bibitem{cochrane-zheng}
Todd Cochrane and Zhiyong Zheng.
\newblock Exponential sums with rational function entries.
\newblock {\em Acta Arithmetica}, 95:67--95, 2000.

\bibitem{EMV_directions_2013}
Daniel El-Baz, Jens Marklof, and Ilya Vinogradov.
\newblock The distribution of directions in an affine lattice: Two-point
  correlations and mixed moments.
\newblock {\em Int. Math. Res. Not.}, 2015(5):1371--1400, January 2015.

\bibitem{el-baz_two-point_2013}
Daniel El-Baz, Jens Marklof, and Ilya Vinogradov.
\newblock The two-point correlation function of the fractional parts of
  {$\sqrt{n}$} is {P}oisson.
\newblock {\em Proc. Amer. Math. Soc.}, 143(7):2815--2828, 2015.

\bibitem{ElkiesMcM04}
Noam~D. Elkies and Curtis~T. McMullen.
\newblock Gaps in {${\sqrt n}\bmod 1$} and ergodic theory.
\newblock {\em Duke Math. J.}, 123(1):95--139, 2004.

\bibitem{flaminio_invariant_2003}
Livio Flaminio and Giovanni Forni.
\newblock Invariant distributions and time averages for horocycle flows.
\newblock {\em Duke Math. J.}, 119(3):465--526, 2003.

\bibitem{marklof_strombergsson_free_path_length_2010}
Jens Marklof and Andreas Str{\"o}mbergsson.
\newblock The distribution of free path lengths in the periodic {L}orentz gas
  and related lattice point problems.
\newblock {\em Ann. of Math.}, 172(3):1949--2033, 2010.

\bibitem{ratner_raghunathans_1991}
Marina Ratner.
\newblock On {R}aghunathan's measure conjecture.
\newblock {\em Ann. of Math. (2)}, 134(3):545--607, 1991.

\bibitem{ratner_raghunathans_1991_1}
Marina Ratner.
\newblock Raghunathan's topological conjecture and distributions of unipotent
  flows.
\newblock {\em Duke Math. J.}, 63(1):235--280, 1991.

\bibitem{sarnak_asymptotic_1981}
Peter Sarnak.
\newblock Asymptotic behavior of periodic orbits of the horocycle flow and
  {E}isenstein series.
\newblock {\em Comm. Pure Appl. Math.}, 34(6):719--739, 1981.

\bibitem{Sinai13}
Ya.~G. Sinai.
\newblock Statistics of gaps in the sequence {$\{\sqrt{n}\}$}.
\newblock In {\em Dynamical systems and group actions}, volume 567 of {\em
  Contemp. Math.}, pages 185--189. Amer. Math. Soc., Providence, RI, 2012.

\bibitem{strombergsson_effective_2013}
Andreas Str{\"o}mbergsson.
\newblock An effective {R}atner equidistribution result for
  {$\mathrm{SL}(2,\mathbb{R})\ltimes\mathbb{R}^2$}.
\newblock {\em Duke Math. J.}, 164(5):843--902, 2015.

\bibitem{strombergsson_venkatesh_2005}
Andreas Str{\"o}mbergsson and Akshay Venkatesh.
\newblock Small solutions to linear congruences and {H}ecke equidistribution.
\newblock {\em Acta Arith.}, 118(1):41--78, 2005.

\bibitem{tenenbaum_introduction_1995}
G{\'e}rald Tenenbaum.
\newblock {\em Introduction to analytic and probabilistic number theory},
  volume~46 of {\em Cambridge Studies in Advanced Mathematics}.
\newblock Cambridge University Press, Cambridge, 1995.
\newblock Translated from the second French edition (1995) by C. B. Thomas.

\end{thebibliography}

\end{document}